\documentclass[12 pt]{amsart}

\usepackage{graphicx}

\usepackage{hyperref}
\usepackage{enumerate}
\usepackage{xcolor}
\usepackage{amsmath}
\usepackage{accents}

\usepackage{amsfonts,amsthm,amssymb,amscd,mathrsfs}

\usepackage{indentfirst}
\usepackage{cite}
\usepackage{latexsym}
\usepackage[dvips]{epsfig}

\usepackage{verbatim}

\renewcommand{\baselinestretch}{1.20}

\usepackage{lipsum}
\usepackage[colorinlistoftodos,textsize=scriptsize]{todonotes}
\usepackage{marginnote}
\usepackage{zref-savepos}

\newcount\todocount

\newcommand{\checkxpos}[3][]{%
\ifdim \zposx{#2}sp < 20000000sp%
\mynote[#1]{#3}%
\else%
\note[#1]{#3}%
\fi%
}

\newcommand{\mytodo}[2][]{%
\zsaveposx{todo\the\todocount}%
\checkxpos[#1]{todo\the\todocount}{#2}%
\global\advance\todocount1\relax
}

\newcommand{\mynote}[2][]{{%
	\let\marginpar\marginnote
	\reversemarginpar
	\renewcommand{\baselinestretch}{0.8}%
	\todo[#1]{#2}}}

\newcommand{\note}[2][]{\renewcommand{\baselinestretch}{0.8}\todo[#1]{#2}}

\newtheorem{theorem}{Theorem}[section]
\newtheorem{remark}{Remark}[section]

\newtheorem{lemma}[theorem]{Lemma}
\newtheorem{pro}[theorem]{Proposition}

\renewcommand{\div}{{\rm div }}

\newcommand{\bt}{\begin{theorem}}
\newcommand{\bl}{\begin{lemma}}
	\newcommand{\el}{\end{lemma}}
\newcommand{\et}{\end{theorem}}

\newcommand{\bn}{\begin{eqnarray}}
\newcommand{\en}{\end{eqnarray}}
\newcommand{\bnn}{\begin{eqnarray*}}
\newcommand{\enn}{\end{eqnarray*}}

\newcommand{\ba}{\begin{aligned}}
\newcommand{\ea}{\end{aligned}}
\newcommand{\be}{\begin{equation}}
\newcommand{\ee}{\end{equation}}

\newcommand{\Bv}{{\boldsymbol{v}}}

\newcommand{\Bu}{{\boldsymbol{u}}}
\newcommand{\Be}{{\boldsymbol{e}}}
\newcommand{\BF}{{\boldsymbol{F}}}

\newcommand{\BU}{\boldsymbol{U}}

\newcommand{\mcA}{\mathcal{A}}
\newcommand{\mcH}{\mathcal{H}}
\newcommand{\mfa}{\mathfrak{a}}

\newcommand{\mfg}{\mathfrak{g}}

\newcommand{\inte}{\int_{-1}^1}

\newcommand{\Z}{\mathbb{Z}}

\newcommand{\Q}{\mathcal{Q}}

\begin{document}

\title[Stability and Uniqueness of Poiseuille Flows]
{Uniqueness and uniform structural stability of Poiseuille flows with large fluxes  in  two-dimensional strips}

\author{Kaijian Sha}
\address{School of mathematical Sciences, Shanghai Jiao Tong University, 800 Dongchuan Road, Shanghai, China}
\email{kjsha11@sjtu.edu.cn}

\author{Yun Wang}
\address{School of Mathematical Sciences, Center for dynamical systems and differential equations, Soochow University, Suzhou, China}
\email{ywang3@suda.edu.cn}

\author{Chunjing Xie}
\address{School of mathematical Sciences, Institute of Natural Sciences,
	Ministry of Education Key Laboratory of Scientific and Engineering Computing,
	and CMA-Shanghai, Shanghai Jiao Tong University, 800 Dongchuan Road, Shanghai, China}
\email{cjxie@sjtu.edu.cn}

\begin{abstract}
	In this paper, we prove the uniform nonlinear structural stability of Poiseuille flows with suitably large flux for the steady Navier-Stokes system in a two-dimensional  strip with arbitrary period.  Furthermore, the well-posedness theory for the Navier-Stokes system is also proved even when the $L^2$-norm of the external force is large. In particular, if the vertical velocity is suitably small where the smallness is independent of the flux, then Poiseuille flow is the unique solution of the steady Navier-Stokes system in the periodic strip. The key point is to establish uniform a priori estimates for the corresponding linearized problem via the boundary layer analysis, where we explore the particular features of odd and even stream functions. The analysis for the even stream function is new, which not only generalizes the previous study for the symmetric flows in \cite{Rabier1}, but also provides an explicit relation between the flux and period.
\end{abstract}

\keywords{Poiseuille flows, steady Navier-Stokes system, two-dimensional, uniform structural stability, periodic.}
\subjclass[2020]{
	35Q30, 35G60,  76D05, 76D03, 76E05}


\maketitle

\section{Introduction and Main Results}
The study on the well-posedness for the steady Navier-Stokes system
\begin{equation}\label{NS}
	\left\{
	\begin{aligned}
		&\boldsymbol{u}\cdot \nabla \boldsymbol{u} +\nabla p=\Delta \boldsymbol{u}+\BF,\\
		&\div~\boldsymbol{u}=0,
	\end{aligned}
	\right.
\end{equation}
in a domain $\Omega$ supplemented with suitable boundary conditions was pioneered by Leray (\hspace{1sp}\cite{Leray}).  In the system \eqref{NS}, $\Bu=(u_1,\cdots, u_N)$ and $\BF=(F_1,\cdots,F_N)(N=2\text{ or }3)$ are  the velocity field and external force, respectively. When the domain $\Omega$ is bounded, the existence of solutions was established in \cite{Leray, Hopf, KPR} and references therein. However, the uniqueness is a much more challenging issue (\hspace{1sp}\cite{Galdi}). In fact, the uniqueness of solutions with large boundary data is an open problem.

If $\Omega$ is an infinitely long channel, in 1950s, Leray proposed a problem to study the well-posedness of the Navier-Stokes system \eqref{NS} supplemented with no-slip boundary conditions and the constraint that the solutions tend to some shear flows at far fields (\hspace{1sp}\cite{Galdi}). This problem is called Leray problem nowadays. In particular, if a two-dimensional channel $\Omega$ tends to the strip $\mathbb{R}\times [-1,1]$ at far field, Leray problem is to show the existence of solutions of Navier-Stokes system, which tend to shear flows for \eqref{NS} satisfying
\begin{equation}\label{BC}
	\boldsymbol{u}=0\text{ on  } y=\pm 1,~~ \ \ \ \int_{-1}^1 u_1(x, y) dy =\Phi.
\end{equation}
The straightforward computations show that the shear flows have the following explicit form
\begin{equation}\label{Poiseuille}
	\boldsymbol{U}=(U(y), 0)=\left(\frac34\Phi(1-y^2), \, 0 \right).
\end{equation}
This flow $\boldsymbol{U}$ is called the Poiseuille flow.
Here  $\Phi\in \mathbb{R}$ is called the flux of the flow.
Without loss of generality, the flux  $\Phi$ is assumed to be nonnegative in this paper.

The first breakthrough for  Leray problem  in infinitely long nozzles was made by Amick (\hspace{1sp}\cite{A1,A2,AF}) and Ladyzhenskaya and Solonnikov (\hspace{1sp}\cite{LS}). It was proved in \cite{A1, LS} that there is a unique solution for  Leray problem as long as the flux is small. Then the convergence rate of the solution (with small flux) at far fields was studied in \cite{AP, Galdi, Horgan, HW} and references therein. One may  also refer to \cite{KP, Morimoto, MF, NP1,NP2, NP3} and references therein for some other studies on the steady Navier-Stokes system in domains with noncompact boundaries. A significant open problem posed in \cite{Galdi} is to prove the existence of solutions for Leray problem when the flux is large. In fact,  it was proved in \cite{LS} that there exists a solution with arbitrary flux of the steady Navier-Stokes system in an infinitely long channel. Therefore, in order to solve Leray problem  in infinitely long channels, one needs only to show that the solutions obtained in \cite{LS} tend to Poiseuille flows at far fields.   To the best of our knowledge, there is no result on the far field behavior of steady solutions with large flux  in a channel except for the axisymmetric solutions in a pipe studied in \cite{WX2}. On the other hand, with the aid of the local compactness of the solutions and blowup techniques,  in order to get the far field behavior of the solutions obtained in \cite{LS}, one needs only to prove a Liouville type theorem, i.e., the global uniqueness of the Poiseuille flows in a straight channel.

As we mentioned before, the uniqueness of solutions of the steady Navier-Stokes system is  very challenging even for flows in bounded domains. The analysis for the solutions of the Navier-Stokes system in unbounded domains is more delicate. For the axisymmetric flows in a pipe,  the  local uniqueness of  solutions with arbitrary flux and even large external force was established in \cite{WX1,WX2}. The problem for two-dimensional flows is much more subtle. When the flow
is symmetric in a strip, Rabier (\hspace{1sp}\cite{Rabier1}) proved the existence and local uniqueness of steady solutions for the Navier-Stokes system with arbitrary flux and sufficiently small external force.
While for the solutions without symmetry in a strip, the existence
of solutions was only proved when neither the flux nor the external force is large \cite{Rabier2}. The well-posedness of steady Navier-Stokes system in a strip with large external force and arbitrary flux is still unsolved.

In fact, the dynamical stability of Poiseuille flows also highly relates to the uniqueness of these flows in the steady setting.
A typical method to study the hydrodynamic stability of Poiseuille flows is the so-called normal mode analysis (\hspace{1sp}\cite{L1,Orszag,DR}), where the analysis on the spectral problem
\begin{equation} \label{spectralpb}
	\left\{ \begin{aligned}
		&s\Bv- \Delta \Bv+\boldsymbol{U} \cdot \nabla \Bv + \Bv \cdot \nabla \boldsymbol{U}  + \nabla P = \BF,\quad \text{in}\,\, \Omega\\
		& {\rm div}~\Bv = 0,\quad \text{in}\,\, \Omega\\
		&	\Bv=0~~\text{on}~\partial\Omega,~~~~\ \ \ \ \ \int_{-1}^1v_1(x,y)dy=0
	\end{aligned}\right.
\end{equation}
plays a crucial role.
Recently, it was proved in \cite{GGN} that a general class of symmetric shear flows  of the two-dimensional incompressible Navier-Stokes equations in a periodic strip are spectrally unstable when the Reynolds number is sufficiently large.
It seems hard to show that $0$ is not a spectrum of the problem \eqref{spectralpb}, which corresponds to the invertibility of linearized problem \eqref{model12}.
Recently, the stability or enhanced dissipations for plane shear flows in a strip under Navier slip boundary conditions  or in the whole plane was studied in \cite{Ding1, Ding2, Conti}.

The uniqueness and structural stability of flows in a periodic strip were studied in \cite{SWX} when  the period is not large.
In this paper, we consider the system \eqref{NS} in a strip $\Omega= \mathbb{T}_{2L\pi} \times[-1,1]$ where the period could be arbitrary  as long as the flux is large. Let $\Bv=\Bu-\boldsymbol{U}$ be  the perturbation around the Poiseuille flow. It  satisfies the following system
\begin{equation}\label{model11}
	\left\{\begin{aligned}
		&-\Delta \Bv+\boldsymbol{U}\cdot \nabla \Bv+\Bv\cdot \nabla \boldsymbol{U}+\nabla P=-\Bv\cdot\nabla\Bv+\BF,\\
		&\mathrm{div}~\Bv=0
	\end{aligned}\right.
\end{equation}
supplemented with the no-slip boundary conditions and the flux constraint
\begin{equation}\label{model11'}
	\Bv=0~~\text{on}~\partial\Omega,~~~~\ \ \ \ \ \int_{-1}^1v_1(x,y)dy=0.
\end{equation}
The crucial point for the analysis on the local uniqueness of  solutions for the problem \eqref{model11}-\eqref{model11'} is to study the associated linearized problem, i.e.,  the linear system
\begin{equation} \label{model12}
	\left\{ \begin{aligned}
		&- \Delta \Bv+\boldsymbol{U} \cdot \nabla \Bv + \Bv \cdot \nabla \boldsymbol{U}  + \nabla P = \BF, \\
		& {\rm div}~\Bv = 0,
	\end{aligned}\right.\ \ \ \mbox{in}\ \Omega,
\end{equation}
with the no-slip boundary conditions and the flux constraint \eqref{model11'}.

The first main result of this paper can be stated as follows.
\begin{theorem}\label{thm1}
	Let $\Omega=\mathbb{T}_{2L\pi} \times[-1,1]$ and assume that $\BF=\BF(x,y)\in L^2(\Omega)$. There exists a uniform constant $\tilde{C}$, such that for all  $\Phi\ge \tilde{C}(1+L)^{63}$, the linearized problem \eqref{model12} and \eqref{model11'} admits a unique solution $\Bv(x,y)\in H^2(\Omega)$ satisfying
	\begin{equation}
		\|\Bv\|_{H^\frac53(\Omega)}\le C\|\BF\|_{L^2(\Omega)}
		\label{estimate11}
	\end{equation}
	and
	\begin{equation}\nonumber
		\|\Bv\|_{H^2(\Omega)}\le C \Phi^{\frac14} \|\BF\|_{L^2(\Omega)},
	\end{equation}
	where $C$ is a uniform constant independent of $L$ and $\Phi$.
\end{theorem}
We have the following remark on Theorem \ref{thm1}.
\begin{remark}
	The key observation of this paper is that the large flux of the flows can provide some good estimates. Since the Reynolds number is proportional to the flux,  the flow is more stable as the Reynolds number becomes larger.
	Theorem \ref{thm1} asserts that if $\Phi\geq \tilde{C}(1+L)^{63 }$, then the linearized problem \eqref{model12} and \eqref{model11'} admits a unique solution. We hope to remove the lower bound, $\tilde{C}(1+L)^{63}$, for the flux in the future.
\end{remark}


Based on Theorem \ref{thm1}, we prove the well-posedness theory of the nonlinear problem.

\begin{theorem}\label{largeforce}
	Let $\Omega=\mathbb{T}_{2L\pi} \times[-1,1]$ and assume that $\BF=\BF(x,y)\in L^2(\Omega)$. For all $\Phi\ge \tilde{C}(1+L)^{63}$, if
	\begin{equation}\nonumber
		\|\BF\|_{L^2(\Omega)}\le \min\left\{1,\sqrt{\frac L2}\right\}\Phi^\frac{1}{32},
	\end{equation}
	then the steady Navier-Stokes system \eqref{NS} supplemented with the no-slip boundary condition and the flux constraint \eqref{BC} admits a unique  solution $\Bu(x,y)$ satisfying
	\begin{equation}\label{est1}
		\left\|\Bu-\BU\right\|_{H^\frac53(\Omega)}\le C \|\BF\|_{L^2(\Omega)},\ \ \ \ \ \|\Bu-\BU\|_{H^2(\Omega)}\le C \Phi^\frac14  \|\BF\|_{L^2(\Omega)},
	\end{equation}
	and
	\begin{equation}\nonumber
		\|u_2\|_{L^2(\Omega)} \leq \min\{1,L^\frac12\} \Phi^{-\frac{5}{12} }.
	\end{equation}
\end{theorem}

\begin{remark}
	Theorem \ref{largeforce} asserts that there exists a unique large solution in a suitably large set  of functions even when the external force is large.	
\end{remark}

In fact, we have the following further result on the uniqueness of  solutions for the Navier-Stokes system in a periodic strip.
\begin{theorem}\label{uniqueness}
	Let $\Omega=\mathbb{T}_{2L\pi} \times[-1,1]$ and assume that $\BF=0$. For all $\Phi\ge \tilde{C}(1+L)^{63}$, if $\Bu$ is a  solution of the problem \eqref{NS}-\eqref{BC} satisfying
	\begin{equation}\label{est6-1}
		\|u_2\|_{H^1(\Omega)}\le \min\left\{1,\sqrt{\frac L2}\right\}\Phi^{\frac{1}{90}},
	\end{equation}
	then $\Bu\equiv\BU$, where $\BU=U(y)\Be_1$ is the Poiseuille flow with $U(y)$ defined in \eqref{Poiseuille}.
	
\end{theorem}

\begin{remark}
	The uniqueness obtained in Theorem \ref{uniqueness} does not require any assumption on $u_1$. Therefore, in order to prove the global uniqueness of Poiseuille flows, one needs only to remove the assumption on $u_2$ later on.
\end{remark}
\begin{remark}
	If we replace $u_2$ by $u_1-U(y)$ in \eqref{est6-1}, then the same conclusion in Theorem \ref{uniqueness} can also be proved.
\end{remark}

\begin{remark}
	An important observation in Theorems \ref{thm1}-\ref{uniqueness} is the constants $C$ and $\tilde{C}$ depend neither on the flux $\Phi$ nor on the period $L$.
\end{remark}

The rest of the paper is organized as follows. In Section \ref{Linear}, Fourier expansion and the stream function formulation are introduced for the nonlinear problem \eqref{model11}-\eqref{model11'}, and the corresponding linearized problem.  Section \ref{sec-res} is devoted to establishing the  uniform a priori estimates independent of the flux $\Phi$ for the linearized problem. Then the uniform nonlinear structural stability of Poiseuille flows is established in Section \ref{sec-nonlinear} with the help of a fixed point theorem. The well-posedness theory of the perturbed problem \eqref{model11}--\eqref{model11'} in the case with the large external force $\BF$ is also proved in Section \ref{sec-nonlinear}. The uniqueness of  solutions (Theorem \ref{uniqueness}) is proved in Section \ref{sec-unique}. Some important lemmas which are used here and there in the paper are collected in Appendix \ref{sec-app}.


\section{Stream function formulation and the linearized problem}\label{Linear}
Suppose that $\Bu=\Bv+\BU$ is a solution of the Navier-Stokes system \eqref{NS}-\eqref{BC} where $\BU=U(y)\Be_1$ is the Poiseuille flow. Then $\Bv$ satisfies the following nonlinear system
\begin{equation}\label{model21}
	\left\{\begin{aligned}
		&-\Delta v_1+U(y)\partial_xv_1+U'v_2+\partial_xP=-(v_1\partial_x+v_2\partial_y)v_1+F_1,\\
		&-\Delta v_2+U(y)\partial_xv_2+\partial_yP=-(v_1\partial_x+v_2\partial_y)v_2+F_2,\\
		&\partial_xv_1+\partial_yv_2=0
	\end{aligned}\right.
\end{equation}
supplemented with the boundary conditions and flux constraint
\begin{equation}\label{model21'}
	\begin{aligned}
		v_1(x,\pm1)=v_2(x,\pm1)=0,~~~~~\ \ \int_{-1}^1 v_1(x,y)\,dy=0, \ \ \ \ x\in [-L\pi, L\pi].
	\end{aligned}
\end{equation}

\subsection{Fourier expansion}
For ease of notation,  we denote $\hat{n}=\frac{n}{L}$  for any $n\in \mathbb{Z}$,  in the rest of the paper.
If the velocity field $\Bv$ is periodic with period $2 L \pi $ in $x$-direction, then it can be written as
\begin{equation}\nonumber
	\Bv=v_1(x,y)\Be_1 +v_2(x,y)\Be_2= \sum\limits_{n \in \mathbb{Z}} \Bv_n e^{i\hat{n}x} =  \sum\limits_{n\in\mathbb{Z}}v_{1,n}(y)e^{i\hat{n}x}\Be_1+v_{2,n}(y)e^{i\hat{n}x}\Be_2,
\end{equation}
where
\begin{equation}\nonumber\begin{aligned}
		v_{i, n}(y) &:=\frac{1}{2L\pi} \int_{-L\pi}^{L\pi} v_i(x,y) e^{-i\hat{n}x} \,dx \ \ \ \ \text{ for }i=1,2.
\end{aligned}\end{equation}
Similarly, the $n$-th  mode of $\BF$ is denoted by
\begin{equation}\nonumber
	\BF_n:= F_{1,n}(y)e^{i\hat{n}x}\Be_1+F_{2,n}(y)e^{i\hat{n}x}\Be_2.
\end{equation}
Since the velocity field $\Bv$ and the force $\BF$ are periodic, $\nabla P$ must be periodic. Hence, one can write
\begin{equation}\nonumber
	\partial_xP=\sum\limits_{n\in \Z}P_{1,n}(y)e^{i\hat{n}x}
	\quad \text{and}
	\quad
	\partial_yP=\sum\limits_{n\in \Z}P_{2,n}(y)e^{i\hat{n}x}.
\end{equation}
Clearly, $P_{1,n}$ and $P_{2,n}$ satisfy
\begin{equation}\nonumber
	i\hat{n}P_{2,n}=P_{1,n}'.
\end{equation}

\subsection{Stream function formulation}\label{sec-stream}
Define
\begin{equation}\nonumber
	\psi_n(y)=\left\{
	\begin{aligned}
		& \psi_0(-1)-\int_{-1}^y v_{1,0}(s)ds,\quad &\text{if}\,\, n= 0,\\
		& -i\frac{1}{\hat{n}}v_{2,n},\quad &\text{if}\,\, n\neq 0,
	\end{aligned}
	\right.
\end{equation}
where $\psi_0(-1)$ is to be determined.
Due to the divergence free property of $\Bv$, it holds that
\be \label{divergencefree}
i\hat{n} v_{1, {n}} + \frac{d}{dy} v_{2, n}=0  \ \ \ \ \text{for any} \ n\in \mathbb{Z}.
\ee
Hence the vorticities of $\Bv_n$ and $\BF_n$ can be written as
\be \label{defomegan}
\omega_n= i \hat{n} v_{2,n}-\frac{d}{d y}v_{1,n}=\left(\frac{d^2}{d y^2}- \hat{n}^2 \right)\psi_n \ \ \ \ \mbox{and}\ \ \ \ f_n= i\hat{n} F_{2,n}-\frac{d}{d y}F_{1,n},
\ee
respectively. Taking curl for the first two equations of \eqref{model21} and rewriting the resulting equation in terms of Fourier series give
\begin{equation}\nonumber
	-U''(y)v_{2,n} +i\hat{n} U(y)\omega_n -\left(\frac{d^2}{d y^2}-\hat{n}^2\right)\omega_n =f_n-\frac{d}{dy}\left(\sum\limits_{m\in \mathbb{Z}} v_{2,n-m}\omega_{m}\right)-i\hat{n}\sum\limits_{m\in \mathbb{Z}} v_{1,n-m}\omega_{m}.
\end{equation}
It follows from the boundary conditions \eqref{model21'} that $v_{1,n}$ and $v_{2,n}$ satisfy
\begin{equation}\label{BC2}
	v_{1,n}(\pm1)=v_{2,n}(\pm1)=0, \text{ and }\int_{-1}^1 v_{1,n}(y)\,dy=0 \quad \text{for }n\in \mathbb{Z}.
\end{equation}
Next, for $n\neq 0$, one has the following equation for $\psi_n$,
\begin{equation}\label{streamform}
	\begin{aligned}
		&-i\hat{n}U''(y)\psi_{n}+i\hat{n}U(y)\left(\frac{d^2}{d y^2}-\hat{n}^2\right)\psi_n-\left(\frac{d^2}{dy^2}-\hat{n}^2\right)^2\psi_n\\
		=&f_n-\frac{d}{dy}\left(\sum\limits_{m\in \mathbb{Z}} v_{2,n-m}\omega_{m}\right)
		-i\hat{n}\sum\limits_{m\in \mathbb{Z}} v_{1,n-m}\omega_{m}.
	\end{aligned}
\end{equation}
The boundary conditions \eqref{BC2} are equivalent to
\begin{equation*}\label{streamBC0}
	\psi_n(\pm1)=\psi_n'(\pm1)=0\quad \text{for }n\in \mathbb{Z}, \, n\neq0.
\end{equation*}
If $n=0$, the boundary conditions \eqref{BC2} together with the equation \eqref{divergencefree} imply that $v_{2,0}=0$ and
\begin{equation*}
	\psi_0(-1)=\psi_0(1), ~\psi_0'(\pm1)=0.
\end{equation*}
Hence $\psi_0$ also satisfies the equation \eqref{streamform}. Obviously, if $n=0$, the equation \eqref{streamform} depends only on the derivatives of $\psi_0$, but not on $\psi_0$ itself. Therefore, without loss of generality, we assume that $\psi_0(1)=\psi_0(-1)=0$ so that the boundary conditions for $\psi_n$ can be written as
\begin{equation}\label{streamBC1}
	\psi_n(\pm1)=\psi_n'(\pm1)=0\quad \text{for }n\in \mathbb{Z}.
\end{equation}

\subsection{Linearized problem} To get the well-posedness of the nonlinear problem \eqref{model21}-\eqref{model21'}, we first study the following linearized problem \begin{equation}\label{model22}
	\left\{\begin{aligned}
		&-\Delta v_1+U \partial_xv_1+U'v_2+\partial_xP=F_1,\\
		&-\Delta v_2+U\partial_xv_2+\partial_yP=F_2,\\
		&\partial_xv_1+\partial_yv_2=0,
	\end{aligned}\right.
\end{equation}
supplemented with the boundary conditions and flux constraint \eqref{model21'}.
Similarly, one can introduce the Fourier expansion and the stream function $\psi_n$ for the solution $\Bv$ of the linearized problem \eqref{model22} and \eqref{model21'}. Then $\psi_n$ satisfies the following fourth order equation
\begin{equation}\label{stream}
	-i\hat{n}U''(y)\psi_{n}+i\hat{n}U(y)\left(\frac{d^2}{d y^2}-\hat{n}^2\right)\psi_n-\left(\frac{d^2}{dy^2}-\hat{n}^2\right)^2\psi_n=f_n.
\end{equation}
Furthermore, the boundary conditions for $\psi_n$ are of the form
\begin{equation}\label{streamBC}
	\psi_n(\pm1)=\psi_n'(\pm1)=0\quad \text{for }n\in \mathbb{Z}.
\end{equation}

\section{Analysis on the linearized problem}\label{sec-res}
The goal of this section is to establish a priori estimates for the linearized problem \eqref{stream} and \eqref{streamBC}. The analysis is divided into three different cases based on the magnitude of the frequency. First, let us recall a priori estimate for the $0$-th mode, which is obtained in \cite{SWX}.
\begin{pro}\label{0mode}\cite[Proposition 3.1.]{SWX}
	For $n=0$, the problem \eqref{stream} and \eqref{streamBC} has a unique solution $\psi_0$ of the form
	\begin{equation}\nonumber
		\psi_{0}(y)=\int_{-1}^y\int_{-1}^\tau\int_{-1}^tF_{1,0}(s)\,dsdtd\tau+\frac{A_1}{6}(y+1)^3+\frac{A_2}{2}(y+1)^2,
	\end{equation}
	where
	\begin{equation}\nonumber
		A_1=\frac32\int_{-1}^1\int_{-1}^\tau\int_{-1}^tF_{1,0}(s)\,dsdtd\tau-\frac32\int_{-1}^1\int_{-1}^tF_{1,0}(s)\,dsdt
	\end{equation}
	and
	\begin{equation}\nonumber
		A_2=\int_{-1}^1\int_{-1}^tF_{1,0}(s)\,dsdt-\frac32\int_{-1}^1\int_{-1}^\tau\int_{-1}^tF_{1,0}(s)\,dsdtd\tau.
	\end{equation}
	Moreover, the corresponding velocity field $\Bv_0=-\psi_0'\Be_1$ satisfies
	\begin{equation}\nonumber
		\|\Bv_0\|_{H^2(\Omega)}\le C_1\|\BF_{0}\|_{L^2(\Omega)},
	\end{equation}
	where $C_1>0$ is a uniform constant independent of flux $\Phi$, $F_{1,0}$, and $L$.
\end{pro}

\subsection{Uniform estimate for the case with large flux and intermediate frequency}\label{secinter}
In this subsection, the uniform a priori estimates for the solution of \eqref{stream}-\eqref{streamBC} with respect to the flux $\Phi$ are established when the flux is large and the frequency is in the intermediate regime. Inspired by \cite{GHM, WX1}, the solutions are decomposed into several parts. The first part is the solution of \eqref{stream} supplemented with the Navier slip boundary conditions and the second part is the associated boundary layer. The other parts are used to recover the equation and the no-slip boundary conditions.

\begin{pro}\label{mediumstream} Assume that  {$\BF_n \in L^2(\Omega)$}.  There exist two uniform constants $\tilde{C}$ and $\epsilon_1$, such that as long as $\Phi \ge \tilde{C}(1+L)^{25}$ and $1\leq |n|\leq \epsilon_1 L\sqrt{\Phi} $, the problem \eqref{stream}-\eqref{streamBC} has a unique solution $\psi_n \in H^3(-1, 1)$. The solution $\psi_n$ can be decomposed into five parts,
	\be \label{decompose}\ba
	\psi_n(y) = &\psi_{n,s}(y) +  b_n^o\left[\psi_{n,BL}^o(y) +\psi_{n,e}^o(y) \right]+b_n^e\left[\psi_{n,BL}^e(y) +\psi_{n,e}^e(y) \right] \\
	& + a_n^o[\psi_{n,p}^o(y)+\psi_{n,r}^o(y)]+a_n^e[\psi_{n,p}^e(y)+\psi_{n,r}^e(y)].
	\ea\ee
	Here $(1)$\ $\psi_{n,s}$ is a solution to the following problem with slip boundary conditions
	\be \label{slip}
	\left\{
	\ba
	&-i\hat{n}U''(y)\psi_{n,s}+ i\hat{n}U(y)\left(\frac{d^2}{d y^2}-\hat{n}^2\right)\psi_{n,s}-\left(\frac{d^2}{d y^2}-\hat{n}^2\right)^2\psi_{n,s}=f_n,\\
	&\psi_{n,s}(\pm1)=\psi_{n,s}''(\pm1)=0.
	\ea\right.\ee
	Moreover, $\psi_{n,s}$ satisfies
	\be \label{est5-3}
	\inte  | \psi_{n,s}' |^2
	+\hat{n}^2| \psi_{n,s} |^2 \, dy \leq C(1+L)^2 |\Phi\hat{n}|^{-1} \inte |\BF_n|^2 \, dy,
	\ee
	\be \label{est5-4}
	\inte |\psi_{n,s}''|^2  +  \hat{n}^2|\psi_{n,s}' |^2  + \hat{n}^4|\psi_{n,s}|^2\, dy
	\leq  C(1+L)^\frac43 |\Phi\hat{n}|^{-\frac23}  \inte |\BF_n|^2 \, dy,
	\ee
	and
	\be \label{est5-5}
	\inte \left|  \psi_{n,s}^{(3)}\right|^2 +
	\hat{n}^2|\psi_{n,s}'' |^2  + \hat{n}^4|\psi_{n,s}'|^2 +\hat{n}^6|\psi_{n,s} |^2 \, dy
	\leq  C\inte |\BF_n|^2 \, dy.
	\ee

	$(2)$ $\psi_{n,BL}^e$ and $\psi_{n,BL}^o$ are the boundary layer profiles
	\be \nonumber
	\psi_{n,BL}^e=\chi^+(y)\psi_{n,BL}^+(y)+\chi^-(y)\psi_{n,BL}^-(y)
	\ \
	\text{and}
	\ \
	\psi_{n,BL}^o=\chi^+(y)\psi_{n,BL}^+(y)-\chi^-(y)\psi_{n,BL}^-(y).
	\ee
	Here
	\be \nonumber
	\psi_{n,BL}^\pm(y)=C_{0,n,\Phi}G_{n,\Phi}(\beta(1\mp y))
	\ee
	with constants  $C_{0, n,\Phi}$ and
	\begin{equation}\label{defbeta}
		\beta=\left|\frac{3\Phi\hat{n}}{2}\right|^\frac13,
	\end{equation}
	and
	$G_{n,\Phi}(\rho)$ is a smooth function, which decays exponentially at infinity and is uniformly bounded in the set
	\be\nonumber
	\mathcal{E}=\{(n,\Phi,\rho):\Phi\ge 1,~1\le |n|\leq L\sqrt{\Phi},~0\le \rho<\infty\}.
	\ee
	$\chi^+(y)=\chi^-(-y)$ are smooth cut-off functions on $[-1,1]$ satisfying that
	\be \label{5-3-5}
	\chi^+(y) = \left\{ \ba  &  1, \ \ \ \ y\geq \frac12,  \\ & 0, \ \ \  \ y\leq \frac14.   \ea  \right.
	\ee
	$\psi_{n,p}^e$ and $\psi_{n,p}^o$ are irrotational flows defined by
	\be\nonumber
	\psi_{n,p}^e=e^{\hat{n}y}+e^{-\hat{n}y}
	\quad
	\text{and}
	\quad
	\psi_{n,p}^o=e^{\hat{n}y}-e^{-\hat{n}y},
	\ee
	respectively.
	The coefficients $a_n^e$, $a_n^o$, $b_n^e$, and $b_n^o$ satisfy
	\be \nonumber
	|b_n^e|+|b_n^o|\leq C(1+L)^\frac56|\Phi\hat{n}|^{-\frac34}\left(\inte |\BF_n|^2\,dy\right)^\frac12
	\ee
	and
	\be\nonumber
	|a_n^e|+|a_n^o|\leq C(1+L)^\frac56|\Phi\hat{n}|^{-\frac34}e^{-|\hat{n}|}\left(\inte |\BF_n|^2\,dy\right)^\frac12.
	\ee

	$(3)$ $\psi_{n,BL}^e+\psi_{n,e}^e$, $\psi_{n,BL}^o+\psi_{n,e}^o$, $\psi_{n,p}^e+\psi_{n,r}^e$, and  $\psi_{n,p}^o+\psi_{n,r}^o$ satisfy the equation \eqref{stream} with $f_n=0$ supplemented with the boundary conditions
	\be \nonumber
	\psi_{n,e}^e(\pm1)=(\psi_{n,e}^e)''(\pm1)=\psi_{n,r}^e(\pm1)=(\psi_{n,r}^e)''(\pm1)=0
	\ee
	and
	\be \nonumber
	\psi_{n,e}^o(\pm1)=(\psi_{n,e}^o)''(\pm1)=\psi_{n,r}^o(\pm1)=(\psi_{n,r}^o)''(\pm1)=0.
	\ee
	
	In conclusion, $\psi_n$ satisfies
	\be \label{est5-6}
	\inte | \psi_n'|^2 + \hat{n}^2|\psi_n|^2 \,dy
	\leq  C(1+L)^2 |\Phi\hat{n}|^{-1}\inte |\BF_n|^2 \, dy
	\ee
	and
	\be \label{est5-7}
	\inte |\psi_n^{(3)}|^2 + \hat{n}^2|\psi_n''|^2 +\hat{n}^4| \psi_n' |^2 +\hat{n}^6|\psi_n|^2 \,dy
	\leq C(1+L)^{\frac53} |\Phi\hat{n}|^{\frac16} \inte |\BF_n|^2 \, dy.
	\ee
	Here $C$ is a uniform constant independent of $\Phi$ and $L$.
\end{pro}

The rest of this subsection devotes to the proof of Proposition \ref{mediumstream}. First, when $f_n$ is odd,  one recalls the a priori estimates obtained in \cite{SWX} for the problem \eqref{slip}.
\begin{lemma}\label{slip-est-o}\cite[Lemma 3.4.]{SWX}
	Assume that {$\BF_n \in L^2(\Omega)$} and  $f_n$ is odd. The problem \eqref{slip} has a unique solution $\psi_{n, s} \in H^3(-1, 1)$, which is also odd and satisfies the estimates
	\be \label{slip-est-o-4}
	\inte  | \psi_{n,s}' |^2
	+\hat{n}^2| \psi_{n,s} |^2 \, dy \leq C|\Phi\hat{n}|^{-\frac43} \inte |\BF_n|^2 \, dy,
	\ee
	\be \label{slip-est-o-5}
	\inte |\psi_{n,s}''|^2  +  \hat{n}^2|\psi_{n,s}' |^2  + \hat{n}^4|\psi_{n,s}|^2 \, dy
	\leq  C|\Phi\hat{n}|^{-\frac23}  \inte |\BF_n|^2 \, dy,
	\ee
	\be \label{slip-est-o-6}
	\inte \left|  \psi_{n,s}^{(3)}\right|^2 +
	\hat{n}^2|\psi_{n,s}'' |^2  + \hat{n}^4|\psi_{n,s}'|^2 +\hat{n}^6|\psi_{n,s} |^2 \, dy
	\leq  C\inte |\BF_n|^2 \, dy,
	\ee
	and
	\begin{equation}\nonumber
		\left|\psi_{n,s}'(\pm1)\right|\le C|\Phi\hat{n}|^{-\frac12}\left(\inte |\BF_n|^2\,dy\right)^\frac12.
	\end{equation}
	Moreover, if $f_n \in L^2(-1, 1)$, it holds that
	\begin{equation}\label{slip-est-o-1}
		\inte |\psi_{n,s}'|^2 + \hat{n}^2|\psi_{n,s}|^2 \,dy \le C|\Phi\hat{n}|^{-\frac53}\inte |f_n|^2\,dy,
	\end{equation}
	\begin{equation}\label{slip-est-o-2}
		\int_{-1}^1 |\psi_{n,s}''|^2 +\hat{n}^2|\psi_{n,s}'|^2+   \hat{n}^4|\psi_{n,s}|^2 \,dy\leq C|\Phi\hat{n}|^{-\frac43}\int_{-1}^1 |f_n|^2\,dy,
	\end{equation}
	and
	\begin{equation}\label{slip-est-o-3}
		\inte \left|\psi_{n,s}^{(3)}\right|^2+\hat{n}^2|\psi_{n,s}''|^2+\hat{n}^4|\psi_{n,s}'|^2+\hat{n}^6|\psi_{n,s}|^2\,dy \le C|\Phi\hat{n}|^{-\frac23}\inte |f_n|^2\,dy.
	\end{equation}
	Here $C$ is a uniform constant.
	
\end{lemma}

In the case that $f_n$ is even, we have the following results, which plays an important role in proving Proposition \ref{mediumstream}.
\begin{lemma}\label{slip-est-e}
	Assume that {$\BF_n\in L^2(\Omega)$}  and $f_n$ is even. The problem \eqref{slip} has a unique solution $\psi_{n, s}\in H^3(-1, 1)$, which is also even and satisfies the  estimates \eqref{est5-3}-\eqref{est5-5} and
	\begin{equation}\nonumber
		\left|\psi_{n,s}'(\pm1)\right|\le C(1+L)^\frac56|\Phi\hat{n}|^{-\frac{5}{12}}\left(\inte |\BF_n|^2\,dy\right)^\frac12.
	\end{equation}
	Moreover, if $f_n \in L^2(-1, 1)$,  $\psi_{n, s}$ satisfies the estimates
	\begin{equation}\label{slip-est-e-1}
		\inte |\psi_{n,s}'|^2 + \hat{n}^2|\psi_{n,s}|^2 \,dy \le C(1+L)^\frac{10}{3}|\Phi\hat{n}|^{-\frac53}\inte |f_n|^2\,dy,
	\end{equation}
	\begin{equation}\label{slip-est-e-2}
		\int_{-1}^1 |\psi_{n,s}''|^2 +\hat{n}^2|\psi_{n,s}'|^2+   \hat{n}^4|\psi_{n,s}|^2 \,dy\leq C(1+L)^\frac{8}{3}|\Phi\hat{n}|^{-\frac43}\int_{-1}^1 |f_n|^2\,dy,
	\end{equation}
	and
	\begin{equation}\label{slip-est-e-3}
		\inte \left|\psi_{n,s}^{(3)}\right|^2+\hat{n}^2|\psi_{n,s}''|^2+\hat{n}^4|\psi_{n,s}'|^2+\hat{n}^6|\psi_{n,s}|^2\,dy \le C(1+L)^\frac{4}{3}|\Phi\hat{n}|^{-\frac23}\inte |f_n|^2\,dy,
	\end{equation}
	where the constant $C$ is independent of $\Phi$ and $L$.
	
\end{lemma}
\begin{remark}
	The estimates for odd and even solutions to the problem \eqref{slip} are slightly different. It seems that the estimates for odd solutions are better, where the estimate \eqref{slip-est-o-4} holds. In fact, the analysis for the two cases are also different. One of the key ideas here is making use of Lemma \ref{lemmaA7} to get the estimate \eqref{5-3-43-new}, since $\psi_{n, s}^\prime$ is odd when $\psi_{n, s}$ is even.
\end{remark}

\begin{proof}[Proof of Lemma \ref{slip-est-e}]
	First, we give some a  priori estimates. Assume that $\psi_{n, s}$ is a smooth even solution to the problem \eqref{slip}.
	Multiplying the equation \eqref{slip} by $\overline{\psi_{n,s}}$ and integrating the resulting equation over $[-1,1]$ yield
	\begin{equation}\label{5-3-9}
		\int_{-1}^1 \hat{n}^4|\psi_{n,s}|^2+2\hat{n}^2|\psi_{n,s}'|^2+|\psi''_{n,s}|^2\,d y=-\Re\int_{-1}^1 f_n\overline{\psi_{n,s}}\,dy+\Im\int_{-1}^1 \hat{n}U'\psi_{n,s}'\overline{\psi_{n,s}}\,dy
	\end{equation}
	and
	\begin{equation}\label{5-3-8}
		\frac{3\Phi\hat{n}}{4}\int_{-1}^1 \hat{n}^2|\psi_{n,s}|^2(1-y^2)+|\psi_{n,s}'|^2(1-y^2)\,d y=-\Im\int_{-1}^1 f_n\overline{\psi_{n,s}}\,dy+\frac{3\Phi\hat{n}}{4}\int_{-1}^1|\psi_{n,s}|^2\, dy.
	\end{equation}
	Similarly, multiplying \eqref{slip} by $\overline{\psi_{n,s}''}$ and integrating the resulting equation over $[-1,1]$ give
	\begin{equation}\label{5-3-36}
		\int_{-1}^1 \overline{\psi_{n,s}''}\left[-i\hat{n}  U''+ i\hat{n} U\left(\frac{d^2}{d y^2}-\hat{n}^2\right)-\left(\frac{d^2}{d y^2}-\hat{n}^2\right)^2\right]\psi_{n,s} \,d y=\int_{-1}^1 \overline{\psi_{n,s}''}f_n\,dy.
	\end{equation}
	It follows from integration by parts and the homogeneous boundary conditions that
	\begin{equation}\nonumber
		\int_{-1}^1-i\hat{n} U''\psi_{n,s}\overline{\psi_{n,s}''}\,dy=\int_{-1}^1i\hat{n} U''|\psi_{n,s}'|^2\,dy
	\end{equation}
	and
	\begin{equation}\nonumber
		\int_{-1}^1 i\hat{n} U \left(\frac{d^2}{d y^2}-\hat{n}^2\right)\psi_{n,s} \overline{\psi_{n,s}''}\,dy=\int_{-1}^1i\hat{n} U|\psi_{n,s}''|^2 +i\hat{n}^3 U|\psi_{n,s}'|^2+i\hat{n}^3 U'\psi_{n,s}\overline{\psi_{n,s}'}\,dy.
	\end{equation}
	Furthermore, one has
	\begin{equation}\nonumber
		\begin{aligned}
			\int_{-1}^1-\left(\frac{d^2}{d y^2}-\hat{n}^2\right)^2\psi_{n,s} \overline{\psi_{n,s}''}\,dy
			=&\int_{-1}^1-\hat{n}^4\psi_{n,s} \overline{\psi_{n,s}''} +2\hat{n}^2|\psi_{n,s}''|^2-\psi_{n,s}^{(4)} \overline{\psi_{n,s}''}\,dy\\
			=&\int_{-1}^1\hat{n}^4|\psi_{n,s}'|^2+
			2\hat{n}^2|\psi_{n,s}''|^2+\left|\psi^{(3)}_n\right|^2\,dy.
		\end{aligned}
	\end{equation}
	Then one can decompose \eqref{5-3-36} into its real and imaginary parts as
	\begin{equation}\label{5-3-40}
		\int_{-1}^1 \hat{n}^4|\psi_{n,s}'|^2+2\hat{n}^2|\psi_{n,s}''|^2+\left|\psi_{n,s}^{(3)}\right|^2\,d y=\Re\int_{-1}^1 f_n\overline{\psi_{n,s}''}\,dy+\Im\int_{-1}^1 \hat{n}^3 U'\psi_{n,s}\overline{\psi_{n,s}'}\,dy
	\end{equation}
	and
	\begin{equation}\label{5-3-41}
		\int_{-1}^1 \hat{n}U''|\psi_{n,s}'|^2+ \hat{n}U|\psi_{n,s}''|^2+ \hat{n}^3 U|\psi_{n,s}'|^2\,d y+\Re\int_{-1}^1 \hat{n}^3 U'\psi_{n,s}\overline{\psi_{n,s}'}\,dy=\Im\int_{-1}^1 f_n\overline{\psi_{n,s}''}\,dy,
	\end{equation}
	respectively. Note that
	\begin{equation}\nonumber
		\Re\int_{-1}^1 \hat{n}^3 U'\psi_{n,s}\overline{\psi_{n,s}'}\,dy=-\frac12\int_{-1}^1 \hat{n}^3 U''|\psi_{n,s}|^2\,dy.
	\end{equation}
	The equation \eqref{5-3-41} can be rewritten as follows,
	\begin{equation}\label{5-3-43}
		\begin{aligned}
			&\frac{3\Phi\hat{n}}{4}\int_{-1}^1 \hat{n}^2|\psi_{n,s}'|^2(1-y^2)+|\psi_{n,s}''|^2(1-y^2)+ \hat{n}^2 |\psi_{n,s}|^2\,d y\\
			=&-\Im\int_{-1}^1 f_n\overline{\psi_{n,s}''}\,dy+\frac{3\Phi\hat{n}}{2}\int_{-1}^1 |\psi_{n,s}'|^2\, dy.
	\end{aligned}\end{equation}
	According to Lemma \ref{lemmaA7}, for $\delta=\min\{\frac{1}{4L^2},\frac{5}{4}\}$, there exists a constant $\delta_1=\min\{\frac{1}{10}\delta, \frac{1}{3}-\frac16\delta\}=\frac{1}{10}\delta$ such that
	\begin{equation}\label{5-3-43-new}
		\begin{aligned}
			\int_{-1}^1 |\psi_{n,s}'|^2\,dy\leq & \left(\frac12 -\delta_1 \right) \int_{-1}^1 |\psi_{n,s}''|^2(1-y^2)\,dy+\delta \int_{-1}^1 |\psi_{n,s}'|^2(1-y^2)\,dy\\
			\leq&\left(\frac12 -\delta_1 \right) \int_{-1}^1 |\psi_{n,s}''|^2(1-y^2)\,dy+\frac{1}{4L^2} \int_{-1}^1 |\psi_{n,s}'|^2(1-y^2)\,dy,
		\end{aligned}
	\end{equation}
	since $\psi_{n,s}$ is even and $\psi_{n,s}'$ is odd. This, together with \eqref{5-3-43}, gives
	\begin{equation}\label{5-3-10}
		\begin{aligned}
			&\int_{-1}^1 \frac12 \hat{n}^2|\psi_{n,s}'|^2(1-y^2)+2\delta_1|\psi_{n,s}''|^2(1-y^2)+ \hat{n}^2 |\psi_{n,s}|^2\,dy\\
			\leq & \left| \frac{4}{3\Phi \hat{n}}\Im\int_{-1}^1 f_n\overline{\psi_{n,s}''}\,dy \right|.
		\end{aligned}
	\end{equation}
	Hence one has
	\begin{equation}\label{5-3-11}
		\begin{aligned}&\int_{-1}^1 \hat{n}^2|\psi_{n,s}'|^2(1-y^2)+|\psi_{n,s}''|^2(1-y^2)+ \hat{n}^2 |\psi_{n,s}|^2+|\psi_{n,s}'|^2\,dy \\
			\leq &C(1+\delta_1^{-1})|\Phi \hat{n}|^{-1}\left|\int_{-1}^1 f_n\overline{\psi_{n,s}''}\,dy\right|\\
			\leq &C(1+L)^2|\Phi \hat{n}|^{-1}\left|\int_{-1}^1 f_n\overline{\psi_{n,s}''}\,dy\right|.
		\end{aligned}
	\end{equation}
	Combining \eqref{5-3-11} and \eqref{5-3-8} yields
	\begin{equation}\label{5-3-12}
		\begin{aligned}
			&\int_{-1}^1 \hat{n}^4|\psi_{n,s}|^2(1-y^2)+ \hat{n}^2|\psi_{n,s}'|^2(1-y^2)+|\psi_{n,s}''|^2(1-y^2)+ \hat{n}^2 |\psi_{n,s}|^2+|\psi_{n,s}'|^2\,d y\\
			\leq &C|\Phi \hat{n}|^{-1}\left|\int_{-1}^1\hat{n}^2 f_n\overline{\psi_{n,s}}\,dy\right|+C(1+L)^2|\Phi \hat{n}|^{-1}\left|\int_{-1}^1f_n\overline{\psi_{n,s}''}\,dy\right|\\
			\leq& C(1+L)^2|\Phi \hat{n}|^{-1}\int_{-1}^1|f_n|(|\psi_{n,s}''|+\hat{n}^2|\psi_{n,s}|)\,dy.
		\end{aligned}
	\end{equation}
	Moreover, it follows from \eqref{5-3-40} and \eqref{5-3-9} that one has
	\begin{equation}\label{5-3-13}
		\int_{-1}^1 \hat{n}^6|\psi_{n,s}|^2+3\hat{n}^4|\psi_{n,s}'|^2+3\hat{n}^2|\psi''_{n,s}|^2+\left|\psi_{n,s}^{(3)}\right|^2\,dy=\Re\int_{-1}^1 f_n\overline{\psi_{n,s}''}\,dy-\Re\int_{-1}^1 \hat{n}^2 f_n\overline{\psi_{n,s}}\,dy.
	\end{equation}
	Thus it holds that
	\begin{equation}\label{5-3-46}
		\int_{-1}^1 \hat{n}^6|\psi_{n,s}|^2+3\hat{n}^4|\psi_{n,s}'|^2+3\hat{n}^2|\psi''_{n,s}|^2+\left|\psi_{n,s}^{(3)}\right|^2\,d y\leq C\int_{-1}^1|f_n|(|\psi_{n,s}''|+\hat{n}^2|\psi_{n,s}|)\,dy.
	\end{equation}
	With the help of \eqref{5-3-12}, \eqref{5-3-46}, and Lemma \ref{weightinequality}, one has
	\begin{equation}\label{5-3-14}
		\begin{aligned}
			&\int_{-1}^1 |\psi_{n,s}''|^2+\hat{n}^4|\psi_{n,s}|^2\,dy\\
			\leq& C\left(\inte |\psi_{n,s}''|^2(1-y^2)\,dy\right)^\frac23\left(\inte |\psi_{n,s}^{(3)}|^2\,dy\right)^\frac13+C\inte |\psi_{n,s}''|^2(1-y^2)\,dy\\
			&+C\left(\inte \hat{n}^4|\psi_{n,s}|^2(1-y^2)\,dy\right)^\frac23\left(\inte \hat{n}^4|\psi_{n,s}'|^2\,dy\right)^\frac13+C\inte \hat{n}^4|\psi_{n,s}|^2(1-y^2)\,dy\\
			\leq&C(1+L)^\frac43|\Phi \hat{n}|^{-\frac23}\int_{-1}^1|f_n|(|\psi_{n,s}''|+\hat{n}^2|\psi_{n,s}|)\,dy\\
			\leq&C(1+L)^\frac43|\Phi \hat{n}|^{-\frac23}\left(\int_{-1}^1|f_n|^2\,dy\right)^\frac12 \left(\int_{-1}^1|\psi_{n,s}''|^2+\hat{n}^4|\psi_{n,s}|^2\,dy\right)^\frac12.
		\end{aligned}
	\end{equation}
	Therefore, it holds that
	\begin{equation}\label{5-3-15}
		\int_{-1}^1 |\psi_{n,s}''|^2+\hat{n}^4|\psi_{n,s}|^2\,dy\leq C(1+L)^\frac83|\Phi \hat{n}|^{-\frac43}\int_{-1}^1|f_n|^2\,dy.
	\end{equation}
	Using \eqref{5-3-15} and Cauchy-Schwarz inequality, it follows from \eqref{5-3-12} and \eqref{5-3-46} that one has
	\begin{equation}\label{5-3-17}
		\int_{-1}^1\hat{n}^2 |\psi_{n,s}|^2+|\psi_{n,s}'|^2\,dy\leq C(1+L)^\frac{10}{3}|\Phi \hat{n}|^{-\frac53}\int_{-1}^1|f_n|^2\,dy
	\end{equation}
	and
	\begin{equation}\label{5-3-18}
		\int_{-1}^1 \hat{n}^6|\psi_{n,s}|^2+3\hat{n}^4|\psi_{n,s}'|^2+3\hat{n}^2|\psi''_{n, s}|^2+\left|\psi_{n,s}^{(3)}\right|^2\,d y\leq C(1+L)^\frac43|\Phi \hat{n}|^{-\frac23}\int_{-1}^1|f_n|^2\,dy.
	\end{equation}

	On the other hand, by virtue of integration by parts and Cauchy-Schwarz inequality, one has
	\begin{equation}\label{5-3-47}\ba
		\left|\Re\int_{-1}^1 f_n\overline{\psi_{n,s}''}\,dy\right|= &~~\left|\Re\int_{-1}^1 i\hat{n} F_{2,n}\overline{\psi_{n,s}''}+F_{1,n}\overline{\psi_{n,s}^{(3)}}\,dy\right|\\
		\leq&~~ C\inte |\BF_n|^2\,dy+\frac12\inte \hat{n}^2|\psi_{n,s}''|^2\,dy+\frac12\inte|\psi_{n,s}^{(3)}|^2\,dy
		\ea\end{equation}
	and
	\begin{equation}\label{5-3-48}\ba
		\left|\Re\int_{-1}^1 \hat{n}^2f_n\overline{\psi_{n,s}}\,dy\right|=&~~\left|\Re\int_{-1}^1 i\hat{n}^3 F_{2,n}\overline{\psi_{n,s}}+\hat{n}^2F_{1,n}\overline{\psi_{n,s}'}\,dy\right|\\
		\leq&~~ C\inte |\BF_n|^2\,dy+\frac12\inte \hat{n}^6|\psi_{n,s}|^2\,dy+\frac12\inte \hat{n}^4|\psi_{n,s}'|^2\,dy.
		\ea\end{equation}
	Combining \eqref{5-3-13} and \eqref{5-3-47}--\eqref{5-3-48} gives \eqref{est5-5}.
	
	Applying integration by parts to the inequality \eqref{5-3-12} and using \eqref{est5-5} yield
	\begin{equation}\label{5-3-55}
		\begin{aligned}
			&\int_{-1}^1 \hat{n}^4|\psi_{n,s}|^2(1-y^2)+ \hat{n}^2|\psi_{n,s}'|^2(1-y^2)+|\psi_{n,s}''|^2(1-y^2)+ \hat{n}^2 |\psi_{n,s}|^2+|\psi_{n,s}'|^2\,d y\\
			\leq &C(1+L)^2|\Phi \hat{n}|^{-1}\left(\left|\int_{-1}^1F_{2,n} i\hat{n}^3\overline{\psi_{n,s}}+F_{1,n}\hat{n}^2\overline{\psi_{n,s}'}\,dy\right|+\left|\int_{-1}^1F_{2,n} i\hat{n}\overline{\psi_{n,s}''}+F_{1,n}\overline{\psi_{n,s}^{(3)}}\,dy\right|\right)\\
			\leq& C(1+L)^2|\Phi \hat{n}|^{-1}\left(\int_{-1}^1|\BF_n|^2\,dy\right)^\frac12\left(\int_{-1}^1|\psi_{n,s}^{(3)}|^2+\hat{n}^2|\psi_{n,s}''|^2+\hat{n}^4|\psi_{n,s}'|^2+\hat{n}^6|\psi_{n,s}|^2\,dy\right)^\frac12\\
			\leq& C(1+L)^2|\Phi \hat{n}|^{-1}\int_{-1}^1|\BF_n|^2\,dy.
		\end{aligned}
	\end{equation}
	Furthermore, it follows from Lemma \ref{weightinequality}, \eqref{est5-5} and \eqref{5-3-55} that one has
	\begin{equation}\label{5-3-53}\begin{aligned}
			&\inte \hat{n}^4|\psi_{n,s}|^2+\hat{n}^2|\psi_{n,s}'|^2+|\psi_{n,s}''|^2\,dy\\
			\leq &C\left(\inte \hat{n}^4|\psi_{n,s}|^2(1-y^2)\,dy\right)^\frac23\left(\inte \hat{n}^4|\psi_{n,s}'|^2\,dy\right)^\frac13+C\inte \hat{n}^4|\psi_{n,s}|^2(1-y^2)\,dy\\
			&+C\left(\inte \hat{n}^2|\psi_{n,s}'|^2(1-y^2)\,dy\right)^\frac23\left(\inte \hat{n}^2|\psi_{n,s}''|^2\,dy\right)^\frac13+C\inte \hat{n}^2|\psi_{n,s}'|^2(1-y^2)\,dy\\
			&+C\left(\inte |\psi_{n,s}''|^2(1-y^2)\,dy\right)^\frac23\left(\inte |\psi_{n,s}^{(3)}|^2\,dy\right)^\frac13+C\inte |\psi_{n,s}''|^2(1-y^2)\,dy\\
			\leq& C(1+L)^\frac43|\Phi\hat{n}|^{-\frac23}\int_{-1}^1|\BF_n|^2\,dy.
	\end{aligned}\end{equation}
	
	We also give the estimate of $|\psi_{n,s}'(1)|$ in terms of $\BF_n$, which plays an important role in recovering the boundary value of $\psi_{n}$. By Lemma \ref{lemmaA2}, it holds that
	\be\nonumber
	\ba
	|\psi_{n,s}'(\pm1)|\leq C\left(\inte |\psi_{n,s}'|^2\,dy\right)^\frac14\left(\inte |\psi_{n,s}''|^2\,dy\right)^\frac14
	\leq  C(1+L)^\frac56|\Phi\hat{n}|^{-\frac{5}{12}}\left(\inte |\BF_n|^2\,dy\right)^\frac12.
	\ea\ee
	
	In fact, the existence  and uniqueness of the solutions follow from the above a priori estimates and standard Galerkin method. For more details, one may refer  to \cite[Appendix B]{SWX}. This finishes the proof of Lemma \ref{slip-est-e}.
\end{proof}

For a general function $f_n$, it can be written as
\be \nonumber
f_n(y)=f_{n}^e(y)+f_{n}^o(y):=\frac{f_n(y)+f_n(-y)}{2}+\frac{f_n(y)-f_n(-y)}{2},
\ee
where $f_{n}^e(y)$ and $f_{n}^o(y)$ are even and odd functions, respectively.  Assume that
$\psi_{n,s}^e(y)$ and $\psi_{n,s}^o(y)$ are solutions to the problems
\be \label{slip-e}
\left\{
\ba
&-i\hat{n}U''(y)\psi_{n,s}^e+ i\hat{n}U(y)\left(\frac{d^2}{d y^2}-\hat{n}^2\right)\psi_{n,s}^e-\left(\frac{d^2}{d y^2}-\hat{n}^2\right)^2\psi_{n,s}^e= f_n^e ,\\
&\psi_{n,s}^e(\pm 1)=(\psi_{n,s}^e)''(\pm 1)=0
\ea
\right.
\ee
and
\begin{equation}\label{slip-o}
	\left\{
	\begin{aligned}
		&-i\hat{n}U''(y)\psi_{n,s}^o+ i\hat{n}U(y)\left(\frac{d^2}{d y^2}-\hat{n}^2\right)\psi_{n,s}^o-\left(\frac{d^2}{d y^2}-\hat{n}^2\right)^2\psi_{n,s}^o= f_n^o ,\\
		&\psi_{n,s}^o(\pm1)=(\psi_{n,s}^o)''(\pm1)=0,
	\end{aligned}
	\right.
\end{equation}
respectively.
Let
\begin{equation*}
	\psi_{n,s}(y)=\psi_{n,s}^e(y)+\psi_{n,s}^o(y).
\end{equation*}
Combining Lemmas \ref{slip-est-o} and \ref{slip-est-e}, one has the following result for problem \eqref{slip}.

\begin{lemma}\label{slip-est} Assume that {$\BF_n\in L^2(\Omega)$}.  The problem \eqref{slip} has a unique solution   $\psi_{n, s} \in H^3(-1, 1)$. The solution $\psi_{n, s}$ satisfies the estimates \eqref{est5-3}-\eqref{est5-5}. Moreover, if $f_n \in L^2(-1, 1)$, the solution $\psi_{n, s}$ satisfies the estimates \eqref{slip-est-e-1}-\eqref{slip-est-e-3}. And consequently,
	\begin{equation}\nonumber
		\left|\psi_{n,s}'(\pm1)\right|\le C(1+L)^\frac56|\Phi\hat{n}|^{-\frac{5}{12}}\left(\inte |\BF_n|^2\,dy\right)^\frac12.
	\end{equation}
\end{lemma}

Now we give the proof of Proposition \ref{mediumstream}.
\begin{proof}[Proof of Proposition \ref{mediumstream}]
	The proof is divided into six steps. We first construct the solution $\psi_n$ in the case that $f_n$ is an even function. The key ideas to handle the case with a general function $f_n$ is given in Step 5. The uniqueness of the solution is proved in the last step.
	
	{\em Step 1. Boundary layer analysis.} In order to recover the no-slip boundary condition, we analyze the associated boundary layer carefully. Define the operators
	\begin{equation}\nonumber
		\mcA_n=i\frac{3\Phi\hat{n}}{4}(1-y^2)-\left(\frac{d^2}{dy^2}-\hat{n}^2\right)\text{ and }
		~~\mcH_n=\frac{d^2}{dy^2}-\hat{n}^2.
	\end{equation}
	Let
	\begin{equation}\nonumber
		\mcA_n^\pm=i\frac{3\Phi\hat{n}}{2}(1\mp y)-\left(\frac{d^2}{dy^2}-\hat{n}^2\right).
	\end{equation}
	They can be regarded as the leading parts of the operator $\mcA_n$ near the boundary $y=\pm1$, respectively.
	
	We look for two  boundary layer functions $\psi_{n,BL}^\pm(y)$, which are the solutions to
	\begin{equation}\label{blayereq}
		\mcA_n^\pm\mcH_n\psi_{n,BL}^\pm=0,
	\end{equation}
	respectively.
	First, let $Ai(z)$ denote the standard Airy function which satisfies
	\begin{equation}\nonumber
		\frac{d^2Ai(z)}{d z^2}-zAi(z)=0 ~~\text{ in }\mathbb{C}.
	\end{equation}
	Define
	\begin{equation}\nonumber
		\widetilde{G}_{n,\Phi}(\rho)=\left\{\begin{aligned}
			& Ai\left(C_+\left(\rho+\frac{2\beta \hat{n}}{3i\Phi}\right)\right),~~\text{if }n> 0,\\
			&Ai\left(C_-\left(\rho+\frac{2\beta \hat{n}}{3i\Phi}\right)\right),~~\text{if }n<0,
		\end{aligned}\right.
	\end{equation}
	where $\beta$ is defined in \eqref{defbeta} and
	\[
	\quad C_\pm=e^{\pm i\frac{\pi}{6}}.
	\]
	One can check that
	\begin{equation}\nonumber
		\mcA_n^\pm\widetilde{G}_{n,\Phi}(\beta(1\mp y))=0.
	\end{equation}
	Next, set
	\begin{equation}\label{5-3-61}
		G_{n,\Phi}(\rho)=\int_\rho^{+\infty}e^{-\frac{|\hat{n}|}{\beta}(\rho-\tau)}
		\int_\tau^{+\infty}e^{-\frac{|\hat{n}|}{\beta}(s-\tau)}\widetilde{G}_{n,\Phi}(s)\,ds\,d\tau.
	\end{equation}
	The straightforward computations show that
	\begin{equation}\nonumber
		\frac{d^2}{d \rho^2}G_{n,\Phi}(\rho)-\frac{\hat{n}^2}{\beta^2}G_{n,\Phi}(\rho)=\widetilde{G}_{n,\Phi}(\rho).
	\end{equation}
	Define
	\begin{equation}\nonumber
		\psi_{n,BL}^\pm=C_{0,n,\Phi}G_{n,\Phi}(\beta(1\mp y))
	\end{equation}
	with
	\begin{equation}\label{5-3-62}
		C_{0,n,\Phi}=
		\left\{\begin{aligned}
			&\frac{1}{G_{n,\Phi}(0)},\quad &\text{if }|G_{n,\Phi}(0)|\ge 1,\\
			&1,\quad &\text{otherwise}.
		\end{aligned}\right.
	\end{equation}
	This implies  that $|\psi^{\pm}_{n,BL}(1)|\le 1$ and $\psi^{\pm}_{n,BL}$ satisfies \eqref{blayereq} for $y\in (-1,1)$.

	Let $\chi ^+\in C^\infty([-1,1])$ be an increasing function satisfying \eqref{5-3-5} and $\chi^-(y)=\chi^+(-y)$. Define
	\begin{equation}\nonumber
		\psi_{n,BL}=\chi^+\psi_{n,BL}^++\chi^-\psi_{n,BL}^-.
	\end{equation}
	
	{\em Step 2. The error term $\psi_{n,e}$.} Suppose that $\psi_{n,e}$ satisfies the following problem
	\begin{equation}\nonumber
		\left\{\begin{aligned}
			&-i\hat{n} U''(y)\psi_{n,e}+i\hat{n}U(y)\left(\frac{d^2}{dy^2}
			-\hat{n}^2\right)\psi_{n,e}-\left(\frac{d^2}{dy^2}-\hat{n}^2\right)^2
			\psi_{n,e}=Q_{n,BL},\\
			&\psi_{n,e}(\pm1)=\psi_{n,e}''(\pm1)=0,
		\end{aligned}\right.
	\end{equation}
	where
	\begin{equation}\nonumber Q_{n,BL}=i\hat{n} U''(y)\psi_{n,BL}-\mcA_n\mcH_n\psi_{n,BL}.
	\end{equation}
	
	According to Lemma \ref{airy-est}, for sufficiently large $\Phi$ such that $\frac12 \beta= \frac12\left|\frac{3\Phi\hat{n}}{2}\right|^\frac13\ge \frac12\left|\frac{3\Phi}{2L}\right|^\frac13\ge R$, where $R$ is the constant appeared in Lemma \ref{airy-est}, one has
	\be\nonumber
	\ba
	&~~\inte  \left|i\hat{n} U''(y)\psi_{n,BL}\right|^2\,dy\leq C|\Phi\hat{n}|^2\int_\frac14^1 \left|\psi_{n,BL}^+\right|^2\,dy\\
	\leq&~~C\left(\int_0^R+\int_R^{\frac34\beta}\right)\beta^5|G_{n,\Phi}(\rho)|^2\,d\rho\\
	\leq&~~C\int_0^R\beta^5\,d\rho+C \int_R^{\frac34\beta}\beta^5e^{-2\rho}\,d\rho\\
	\leq&~~C\beta^5.
	\ea\ee
	Using the fact  $\mcA_n^+\mcH_n\psi_{n,BL}^+=0$ yields
	\be\nonumber
	\ba
	-\mcA_n \mcH_n (\chi^+\psi_{n,BL}^+)=&~~
	\mcA_n^+\mcH_n((1-\chi^+)\psi_{n,BL}^+)+(\mcA_n^+\mcH_n-\mcA_n\mcH_n)(\chi^+\psi_{n,BL}^+)\\
	=&~~\left(i\frac{3\Phi\hat{n}}{2}(1-y)-\frac{d^2}{dy^2}+\hat{n}^2\right)\left(\frac{d^2}{dy^2}-\hat{n}^2\right)((1-\chi^+)\psi_{n,BL}^+)\\
	&~~+i\frac{3\Phi\hat{n}}{4}(1-y)^2\left(\frac{d^2}{dy^2}-\hat{n}^2\right)(\chi^+\psi_{n,BL}^+).
	\ea\ee
	Noting that $\chi^+\psi_{n,BL}^+$ vanishes on $[-1,0]$, it suffices to estimate $\mcA_n\mcH_n(\chi^+\psi_{n,BL}^+)$ on $[0,1]$. According to Lemma \ref{airy-est}, one has
	\be\nonumber
	\ba
	&~~
	\int_0^1\left|\frac{3\Phi\hat{n}}{2 }(1-y)\left(\frac{d^2}{dy^2}-\hat{n}^2\right)((1-\chi^+)\psi_{n,BL}^+)\right|^2\,dy\\
	\le&~~C
	\int_0^\frac12\beta^6\left( |(\psi_{n,BL}^+) ''|^2+| (\psi_{n,BL}^+)'|^2
	+|\psi_{n,BL}^+ |^2+\hat{n}^4|\psi_{n,BL}^+ |^2\right)\,dy\\
	\le&~~C
	\int_{\frac12\beta}^{\beta}\beta^9|G_{n,\Phi}''(\rho)|^2+\beta^7|G_{n,\Phi}'(\rho)|^2
	+\beta^5|G_{n, \Phi }(\rho)|^2+\beta^9|G_{n,\Phi}(\rho)|^2\,d\rho.
	\ea
	\ee
	Hence one has
	\be\ba
	\int_0^1\left|\frac{3\Phi\hat{n}}{2 }(1-y)\left(\frac{d^2}{dy^2}-\hat{n}^2\right)((1-\chi^+)\psi_{n,BL}^+)\right|^2\,dy
	\le&~~C
	\int_{\frac12\beta}^{\beta}(\beta^9+\beta^7
	+\beta^5)e^{-2\rho}\,d\rho\\
	\le&~~C e^{-\beta}(\beta^{10}+\beta^8
	+\beta^6)\le C.
	\ea\ee
	On the other hand,
	\be\nonumber
	\ba
	&~~\int_0^1\left|(\mcA_n^+\mcH_n-\mcA_n\mcH_n)(\chi^+\psi_{n,BL}^+)\right|^2\,dy\\
	=&~~
	\int_\frac14^1\left|\frac{3\Phi\hat{n}}{4}(1-y)^2\left(\frac{d^2}{dy^2}-\hat{n}^2\right)(\chi^+ \psi_{n,BL}^+)\right|^2\,dy\\
	\le&~~C
	\int_\frac14^1\beta^6(1-y)^4 \left(| \left( \psi_{n,BL}^{+} \right)''|^2+|(\chi^+)'|^2|(\psi_{n,BL}^+)'|^2
	+|(\chi^+)''|^2|\psi_{n,BL}^+|^2+\hat{n}^4|\psi_{n,BL}^+|^2\right)\,dy\\
	\le&~~C
	\int_0^{\frac34\beta}\rho^4\left(\beta^5|G_{n,\Phi}''|^2+\beta^5|G_{n,\Phi}|^2\right)\,d\rho+C
	\int_{\frac12\beta}^{\frac34\beta}\rho^4\left(\beta^3|G_{n,\Phi}'|^2
	+\beta|G_{n,\Phi}|^2\right)\,d\rho\\
	\le&~~C\beta^5\int_0^R\rho^4\,d\rho
	+C\beta^5\int_R^{\frac34\beta}\rho^4e^{-2\rho}\,d\rho+C\left(\beta^3
	+\beta\right)
	\int_{\frac12\beta}^{\frac34\beta}\rho^4e^{-2\rho}\,d\rho\\
	\leq&~~C\beta^5.
	\ea\ee
	Hence it follows from Lemma \ref{slip-est} that
	\begin{equation}\label{5-3-74}
		\inte \hat{n}^2|\psi_{n,e}|^2+|\psi_{n,e}'|^2\,dy \le C(1+L)^\frac{10}{3}|\Phi\hat{n}|^{-\frac53}\beta^5\le C(1+L)^\frac{10}{3},
	\end{equation}
	\begin{equation}\label{5-3-75}
		\int_{-1}^1 \hat{n}^4|\psi_{n,e}|^2+\hat{n}^2|\psi_{n,e}'|^2+|\psi_{n,e}''|^2\,d y\leq C(1+L)^\frac{8}{3}|\Phi\hat{n}|^{-\frac43}\beta^5\le C(1+L)^\frac{8}{3}|\Phi\hat{n}|^\frac13,
	\end{equation}
	and
	\begin{equation}\label{5-3-76}
		\inte \left|\psi_{n,e}^{(3)}\right|^2+\hat{n}^2|\psi_{n,e}''|^2+\hat{n}^4|\psi_{n,e}'|^2+\hat{n}^6|\psi_{n,e}|^2\,dy \leq C(1+L)^\frac{4}{3}|\Phi\hat{n}|^{-\frac23}\beta^5\le C(1+L)^\frac{4}{3}|\Phi\hat{n}|.
	\end{equation}
	
	{\em Step 3. The irrotational flow $\psi_{n,p}$ and the associated remainder term $\psi_{n,r}$.} Denote
	\begin{equation}\nonumber
		\psi_{ n ,p}=e^{\hat{n}y}+e^{-\hat{n}y}.
	\end{equation}
	Let $\psi_{n,r}$ be the solution of the following problem,
	\begin{equation}\nonumber
		\left\{\begin{aligned}
			&-i\hat{n} U''(y)\psi_{n,r}+i\hat{n}U(y)\left(\frac{d^2}{dy^2}
			-\hat{n}^2\right)\psi_{n,r}-\left(\frac{d^2}{dy^2}-\hat{n}^2\right)^2
			\psi_{n,r}=i\hat{n} U''(y)\psi_{n,p},\\
			&\psi_{n,r}(\pm1)=\psi_{n,r}''(\pm1)=0.
		\end{aligned}\right.
	\end{equation}
	Let $a_n$ and $b_n$ be constants satisfying
	\begin{equation}\label{recoverBC1}
		\left\{\begin{aligned}
			&b_n\psi_{n,BL}^+(1)+a_n\psi_{n,p}(1)=0,\\
			&\psi_{n,s}'(1)+b_n[(\psi_{n,BL}^+)'(1)+\psi_{n,e}'(1)]+a_n[\psi_{n,p}'(1)+\psi_{n,r}'(1)]=0.
		\end{aligned}\right.
	\end{equation}
	This implies that $\psi_n =\psi_{n,s}+b_n (\psi_{n,BL}+\psi_{n,e})+a_n(\psi_{n, p}+\psi_{n,r})$ satisfies the no-slip boundary condition at $y=1$.
	Furthermore, if $f_n$ is even with respect to $y$, the associate solution $\psi_{n,s}$ is also even. Similarly, one can verify that all the components $\psi_{n,BL}$, $\psi_{n,e}$, $\psi_{n,p}$, and $\psi_{n,r}$ are all even. Thus we also have
	\begin{equation}\label{recoverBC-1}
		\left\{\begin{aligned}
			&b_n\psi_{n,BL}^-(-1)+a_n\psi_{n,p}(-1)=0,\\
			&\psi_{n,s}'(-1)+b_n[(\psi_{n,BL}^-)'(-1)+\psi_{n,e}'(-1)]+a_n[\psi_{n,p}'(-1)+\psi_{n,r}'(-1)]=0.
		\end{aligned}\right.
	\end{equation}
	Therefore, $\psi_n =\psi_{n,s}+b_n (\psi_{n,BL}+\psi_{n,e})+a_n(\psi_{n, p}+\psi_{n,r})$ satisfies \eqref{streamBC} at $y=\pm 1$.
	Solving the linear system \eqref{recoverBC1} gives
	\begin{equation}\label{5-3-79}
		a_n=-\frac{\psi_{n,BL}^+(1)}{\psi_{n,p}(1)}b_n\quad\text{and}\quad
		b_n=\frac{\psi_{n,s}'(1)}{\psi_{n,BL}^+ (1)\frac{\psi_{n,p}'(1)+\psi_{n,r}'(1)}{\psi_{n,p}(1)}
			-(\psi_{n,BL}^+)'(1)-\psi_{n,e}'(1)}.
	\end{equation}
	In order to get the estimates for $a_n$ and $b_n$, one of the key issues is to estimate $(\psi_{n,BL}^+)'(1)$. The straightforward computations show
	\begin{equation}\nonumber
		\begin{aligned}
			(\psi_{n,BL}^+)'(1)=-\beta C_{0,n,\Phi}\frac{d G_{n,\Phi}}{d \rho}(0)
			=\beta C_{0,n,\Phi}\left(\frac{|\hat{n}|}{\beta }G_{n,\Phi}(0)+C_-\int_\ell e^{-\lambda z}Ai\left(z+\lambda^2\right)\,dz\right),
		\end{aligned}
	\end{equation}
	where $\lambda=\frac{|\hat{n}|}{\beta }C_-$ and $\ell$ is the contour $\ell:=\left\{re^{i\frac{\pi}{6}}|r\ge 0\right\}$. Note that $\arg\lambda=-\frac{\pi}{6}$ and
	\begin{equation}\nonumber
		|\lambda|=\frac{|\hat{n}|}{\beta }=\left(\frac23\right)^\frac13(\Phi^{-\frac12}|\hat{n}|)^\frac23.
	\end{equation}
	Choose  $\epsilon_1\in (0,1)$ such that
	\begin{equation}\nonumber
		|\lambda|\le \min \left\{\epsilon,\ \frac{1}{48}\tilde{C}_0 \right\}
	\end{equation}
	as long as $|\hat{n}|\le \epsilon_1\Phi^\frac12$,
	where $\epsilon$ and $\tilde{C}_0$ are  the  constants indicated in Lemma \ref{airy-est}.
	One can apply Lemma \ref{airy-est} to obtain
	\begin{equation}\label{5-3-82}
		\begin{aligned}
			|(\psi_{n,BL}^+)'(1)|=&\left|\beta C_{0,n,\Phi}\left(\frac{|\hat{n}|}{\beta }G_{n,\Phi}(0)+C_-\int_\ell e^{-\lambda z}Ai\left(z+\lambda^2\right)\,dz\right)\right|\\
			\ge&\left|\beta C_{0,n,\Phi}\int_\ell e^{-\lambda z}Ai\left(z+\lambda^2\right)\,d z\right|-\left|\beta C_{0,n,\Phi}\frac{|\hat{n}|}{\beta }G_{n,\Phi}(0)\right|.
		\end{aligned}
	\end{equation}
	Hence it holds that
	\begin{equation}\label{kappa}
		\begin{aligned}
			|(\psi_{n,BL}^+)'(1)|
			\ge & \beta \tilde{C}_0\left(\frac16-\frac{|\hat{n}|}{\tilde{C}_0\beta }\right)\ge   \frac{1}{12}\beta \tilde{C}_0=:\kappa\beta .
	\end{aligned}\end{equation}
	
	For the term $\psi_{n,e}'(1)$, it follows from \eqref{5-3-74}-\eqref{5-3-75} and Lemma \ref{lemmaA2} that
	\begin{equation}\label{5-3-83}
		\begin{aligned}
			|\psi_{n,e}'(1)|\le&C\left(\inte|\psi_{n,e}'|^2\,dy\right)^\frac14\left(\inte |\psi_{n,e}''|^2\,dy\right)^\frac14\le C(1+L)^\frac32|\Phi\hat{n}|^\frac{1}{12}.
		\end{aligned}
	\end{equation}
	On the other hand, using Lemmas \ref{lemmaA2} and \ref{slip-est} again gives
	\begin{equation}\nonumber
		\begin{aligned}
			|\psi_{n,r}'(1)|\leq &C \left(\inte |\psi_{n,r}'|^2\,dy \right)^\frac14\left(\inte  |\psi_{n,r}''|^2\,dy \right)^\frac14\\
			\leq& C(1+L)^\frac32|\Phi\hat{n}|^{-\frac34} \left(\inte \left|i\hat{n} U''(y)\psi_{n,p}\right|^2\,dy \right)^\frac12\\
			\leq& C(1+L)^\frac32|\Phi\hat{n}|^\frac14 \left(\inte \left|\psi_{n,p}\right|^2\,dy \right)^\frac12\\
			\leq& C(1+L)^\frac32 |\Phi\hat{n}|^\frac14  |\hat{n} |^{-\frac12} \psi_{n,p}(1) \\
			\leq & C (1+L)^2 |\Phi \hat{n} |^{\frac14} \psi_{n, p}(1) .
		\end{aligned}
	\end{equation}
	Then there exists a uniform constant $\tilde{C}$ such that if $\Phi \ge \tilde{C}(1+L)^{25}$, then one has
	\begin{equation*}
		(1+L)^2 \leq \left(\frac{\Phi}{\tilde{C}(1+L)}\right)^\frac{1}{12}\leq (\tilde{C}^{-1 }|\Phi\hat{n}|)^\frac{1}{12}
	\end{equation*}
	and
	\begin{equation}\label{5-3-85}
		\begin{aligned}
			\left|\psi_{n,BL}^+(1)\frac{\psi_{n,p}'(1)+\psi_{n,r}'(1)}{\psi_{n,p}(1)}\right|\leq C|\hat{n}|+C(1+L)^2 |\Phi\hat{n}|^\frac14\leq \frac\kappa 4\beta,
		\end{aligned}
	\end{equation}
	provided that $\epsilon_1$ is sufficiently small, where $\kappa$ is the constant defined in \eqref{kappa}.  Combining \eqref{5-3-79}-\eqref{5-3-85}                                                                                                                     yields
	\begin{equation}\label{5-3-86}
		\begin{aligned}
			|b_n|\leq C\beta^{-1}|\psi_{n,s}'(1)|\leq C(1+L)^\frac56|\Phi\hat{n}|^{-\frac34}\left(\inte |\BF_n|^2\,dy\right)^\frac12
		\end{aligned}
	\end{equation}
	and
	\begin{equation}\label{5-3-87}
		\begin{aligned}
			|a_n|\leq &C(1+L)^\frac56|\Phi\hat{n}|^{-\frac34}|\psi_{n,p}(1)|^{-1}\left(\inte |\BF_n|^2\,dy\right)^\frac12.
		\end{aligned}
	\end{equation}
	
	Next one can obtain the estimates of the boundary layer $b_n\psi_{n,BL}$. According to Lemma \ref{airy-est}, using $|\hat{n}|\leq \epsilon_1|\Phi|^\frac12$ yields
	\begin{equation}\label{5-3-88}
		\begin{aligned}
			&|b_n|^2\inte \hat{n}^2|\psi_{n,BL}|^2+\left|\psi_{n,BL}'\right|^2\,dy\\
			\le&C|b_n|^2\int_\frac14^1\hat{n}^2|\psi_{n,BL}^+|^2+\left|\psi_{n,BL}^+\right|^2+\left|(\psi_{n,BL}^+)'\right|^2\,dy\\
			\le&C|b_n|^2\int_0^{\frac34\beta }\left(\hat{n}^2|G_{n,\Phi}(\rho)|^2+|G_{n,\Phi}(\rho)|^2+\left|\frac{d}{d\rho}G_{n,\Phi}(\rho)\right|^2\beta ^2\right)\beta^{-1}\,d \rho\\
			\le&C (1+L )^{\frac53} | \Phi\hat{n}|^{-\frac32}\inte |\BF_n|^2\,dy\left(\int_0^R\beta +\beta ^{-1}\,d\rho+\int_R^{\frac34\beta }e^{-2\rho}(\beta +\beta ^{-1})\,d\rho\right)\\
			\le&C(1+L)^\frac53|\Phi\hat{n}|^{-\frac76}\inte |\BF_n|^2\,dy\\
			\leq & C (1 + L)^2 |\Phi \hat{n}|^{-1} \inte |\BF_n|^2 \, dy.
		\end{aligned}
	\end{equation}
	Similar computations give
	\begin{equation}\label{5-3-89}
		\begin{aligned}
			&|b_n|^2\inte\left|\psi_{n,BL}^{(3)}\right|^2+\hat{n}^2\left|\psi_{n,BL}''\right|^2+\hat{n}^4\left|\psi_{n,BL}'\right|^2+\hat{n}^6\left|\psi_{n,BL}\right|^2\,dy\\
			\le & C(1+L)^\frac53|\Phi\hat{n}|^{\frac16}\inte |\BF_n|^2\,dy.
		\end{aligned}
	\end{equation}
	
	It follows from \eqref{5-3-74}-\eqref{5-3-76} and \eqref{5-3-86} that one has
	\begin{equation}\label{5-3-90} \begin{aligned}
			|b_n|^2\inte \hat{n}^2|\psi_{n,e}|^2+|\psi_{n,e}'|^2\,dy  & \le C(1+L)^5|\Phi\hat{n}|^{-\frac32}\inte |\BF_n|^2\,dy \\
			& \leq C (1+L)^2 |\Phi \hat{n}|^{-1} \inte |\BF_n|^2 \, dy
		\end{aligned}
	\end{equation}
	and
	\begin{equation}\label{5-3-91} \begin{aligned}
			& |b_n|^2\inte \left|\psi_{n,e}^{(3)}\right|^2+\hat{n}^2|\psi_{n,e}''|^2+\hat{n}^4|\psi_{n,e}'|^2+\hat{n}^6|\psi_{n,e}|^2\,dy  \\
			\leq  &C(1+L)^3|\Phi\hat{n}|^{-\frac12}\inte |\BF_n|^2\,dy \\
			\leq & C (1+L)^{\frac53} |\Phi \hat{n} |^{\frac16} \inte |\BF_n|^2 \, dy.
		\end{aligned}
	\end{equation}
	Meanwhile, the straightforward computations yield
	\begin{equation}\label{5-3-92}\begin{aligned}
			|a_n|^2\inte \hat{n}^2|\psi_{n,p}|^2+|\psi_{n,p}'|^2\,dy
			\le &C|a_n|^2\inte \hat{n}^2|e^{\hat{n}y}+e^{-\hat{n}y}|^2+\hat{n}^2|e^{\hat{n}y}-e^{-\hat{n}y}|^2\,dy\\
			\le &C|a_n|^2\inte \hat{n}^2(e^{2\hat{n}y}+e^{-2\hat{n}y})\,d y\\
			\le &C(1+L)^\frac53|\Phi\hat{n}|^{-\frac76}\inte |\BF_n|^2\,dy\\
			\leq & C (1 + L)^2 |\Phi \hat{n}|^{-1} \inte |\BF_n|^2 \, dy
	\end{aligned}\end{equation}
	and
	\be \label{5-3-93}\ba
	|a_n|^2 \inte \left| \psi_{n,p}^{(3)}\right|^2 +\hat{n}^2|\psi_{n,p}''|^2 + \hat{n}^4 |\psi_{n,p}'|^2 + \hat{n}^6| \psi_{n,p} |^2 \,dy
	\leq C(1+L)^\frac53|\Phi\hat{n}|^{\frac16}\inte |\BF_n|^2\,dy.
	\ea\ee
	Finally, if follows from Lemma \ref{slip-est} that
	\begin{equation}\label{5-3-94}\ba
		~~|a_n|^2\inte \hat{n}^2|\psi_{n,r}|^2+|\psi_{n,r}'|^2\,dy
		\leq &~~ C|a_n|^2 (1+L)^\frac{10}{3}|\Phi\hat{n}|^{-\frac53}\inte \left|i\hat{n} U''(y)\psi_{n,p}\right|^2\,dy\\
		\leq&~~ C(1+L)^5|\Phi\hat{n}|^{-\frac76}\inte |\BF_n|^2\,dy \\
		\leq &~~ C (1 + L)^2 |\Phi \hat{n}|^{-1} \inte |\BF_n|^2 \, dy
		\ea\end{equation}
	and
	\begin{equation}\label{5-3-95}\ba
		&~~|a_n|^2\inte \left|\psi_{n,r}^{(3)}\right|^2+\hat{n}^2|\psi_{n,r}''|^2+\hat{n}^4|\psi_{n,r}'|^2+\hat{n}^6|\psi_{n,r}|^2\,dy \\
		\leq&~~ C|a_n|^2 (1+L)^\frac43|\Phi\hat{n}|^{-\frac23}\inte \left|i\hat{n} U''(y)\psi_{n,p}\right|^2\,dy\\
		\leq&~~ C(1+L)^3|\Phi\hat{n}|^{-\frac16}\inte |\BF_n|^2\,dy
		\leq  C (1+L)^{\frac53} |\Phi \hat{n}|^{\frac16} \inte |\BF_n|^2 \, dy.
		\ea\end{equation}
	Combining the estimates \eqref{5-3-88}-\eqref{5-3-95} and Lemma \ref{slip-est} gives the estimates \eqref{est5-6} and \eqref{est5-7}.
	
	{\it Step 5. The case with general $f_n$.}  From Steps 1-4, if $f_n$ is an even function,  we can construct the solution $\psi_n$  to the problem \eqref{stream}-\eqref{streamBC} in the form of \eqref{decompose} with $a_n^o=b_n^o=0$. Similarly,
	if $f_n$ is an odd function, one can construct the solution $\psi_n$  to the problem \eqref{stream}-\eqref{streamBC} in the form of \eqref{decompose} with
	$a_n^e=b_n^e=0$.
	Since the boundary conditions \eqref{recoverBC1} and \eqref{recoverBC-1} can be satisfied simutaneously, one can also get the estimates \eqref{5-3-86}-\eqref{5-3-95} for $a_n^o$, $b_n^o$, $\psi_{n,BL}^o$, $\psi_{n,e}^o$,  and $\psi_{n,r}^o$ in the same way. Therefore, for a given general $f_n$, the solutions associated with $f_{n}^e$ and $f_{n}^o$ satisfy the estimates \eqref{est5-6}-\eqref{est5-7}. Hence the solution $\psi_n$ satisfies \eqref{est5-6}-\eqref{est5-7}.
	
	{\it Step 6. Uniqueness of the solution.} Note that a solution $\psi_n$ of the problem \eqref{stream}-\eqref{streamBC} has been constructed during Steps 1-5. It remains to prove the uniqueness of the solution, which is achieved by the Fredholm alternative theorem. Denote $\sigma = \frac{9}{16}\Phi^2 $ and define
	\begin{equation*}
		\mathcal{L} \psi = -i\hat{n} U''(y) +  i\hat{n} U(y) \left(\frac{d^2}{dy^2} - \hat{n}^2\right) \psi - \left( \frac{d^2}{dy^2} - \hat{n}^2 \right)^2 \psi\,\,\text{and}\,\, \mathcal{L}_\sigma =\mathcal{L}-\sigma I.
	\end{equation*}
	Define the bilinear functional $\mfa_\sigma (\cdot, \cdot)$ on $H^2_0(-1, 1)$ as follows,
	\begin{equation*}
		\mfa_\sigma (\psi, \varphi) = \int_{-1}^1  -i\hat{n} U''(y) \psi \overline{\varphi} + i \hat{n} U(y) \left(\frac{d^2}{dy^2} - \hat{n}^2 \right) \psi  \overline{\varphi}
		- \left(\frac{d^2}{dy^2} - \hat{n}^2 \right) \psi  \left(\frac{d^2}{dy^2} - \hat{n}^2 \right) \overline{\varphi}-\sigma \psi \overline{\varphi} \, dy .
	\end{equation*}
	It is clear that $\mfa_\sigma (\cdot, \cdot)$ is a bounded bilinear functional on $H_0^2(-1,1)$. {On the other hand, for any $\psi\in H_0^2(-1,1)$, one uses Cauchy-Schwarz inequality to obtain
		\begin{equation*}
			\begin{aligned}
				|\mfa_\sigma (\psi,\psi)|\ge &|\Re \mfa_\sigma (\psi,\psi)|=\left|\int_{-1}^1|\psi''|^2+2\hat{n}^2|\psi'|^2+\hat{n}^4|\psi|^2 +\sigma |\psi|^2\,dy-\Im \int_{-1}^1 \hat{n} U'\psi'\overline{\psi}\,dy\right|\\
				\ge &\int_{-1}^1|\psi''|^2+2\hat{n}^2|\psi'|^2+\hat{n}^4|\psi|^2 +\sigma |\psi|^2\,dy -\frac32\Phi|\hat{n}| \int_{-1}^1|\psi'\overline{\psi}|\,dy\\
				\ge& \int_{-1}^1|\psi''|^2+2\hat{n}^2|\psi'|^2+\hat{n}^4|\psi|^2 +\sigma |\psi|^2\,dy -\int_{-1}^1\hat{n}^2|\psi'|^2+\frac{9}{16}\Phi^2 |\psi|^2\,dy\\
				=&\int_{-1}^1|\psi''|^2+ \hat{n}^2|\psi'|^2+\hat{n}^4|\psi|^2 \,dy.
			\end{aligned}
		\end{equation*}
		Hence $\mfa_\sigma (\cdot,\cdot)$ is coercive in $H_0^2(-1,1)$.} For every $ \mfg \in H^{-1}(-1, 1)$, denote $\langle \mfg, \varphi \rangle $ by the linear functional $\mfg$ acting on $\varphi \in H_0^2(-1, 1)$. Hence one has
	\[
	|\langle \mfg, \varphi \rangle|\leq \|\mfg\|_{H^{-1}(-1,1)}\|\varphi\|_{H_0^2(-1,1)}.
	\]
	It follows from Lax-Milgram theorem that for any $\mfg\in H^{-1}(-1,1)$,
	there exists a unique  $\psi \in H_0^2(-1, 1)$ such that
	\begin{equation*}
		\mfa_\sigma (\psi, \varphi)  =  \langle \mfg, \varphi\rangle  \ \ \ \ \ \ \ \mbox{for any}\ \varphi \in H^2_0(-1, 1),
	\end{equation*}
	which means that $\psi$ is a weak solution to the problem
	\begin{equation} \label{augmented}
		\left\{
		\ba
		&\mathcal{L}_\sigma \psi = \mfg , \ \ \ \ \ \mbox{in}\ (-1, 1), \\
		&\psi(\pm 1) = \psi^{\prime}(\pm 1) = 0.
		\ea
		\right.
	\end{equation}
	We denote $\psi=\mathcal{L}_\sigma^{-1}(\mfg)$. Moreover, it holds that
	\begin{equation*}
		\|     \psi     \|_{H^2(-1, 1) } \leq C \| \mfg\|_{H^{-1}(-1, 1)}.
	\end{equation*}
	According to the standard regularity theory for elliptic systems, $\psi \in H^3(-1, 1)$ and
	{
		\begin{equation*}
			\ba
			&\| \psi\|_{H^3(-1, 1)} \\
			\leq & C \left\|-i \hat{n}U''(y) \psi  + i \hat{n} U(y) \left(\frac{d^2}{dy^2} - \hat{n}^2 \right) \psi +2\hat{n}^2\psi''-\hat{n}^4\psi-\sigma \psi \right\|_{L^2(-1, 1)}
			+ C \|\mfg\|_{H^{-1}(-1, 1)}\\
			\leq & C \|\mfg\|_{H^{-1}(-1, 1)}.
			\ea \end{equation*}
	}
	This implies $\mathcal{L}_\sigma^{-1}$ is a bounded operator from $H^{-1}(-1,1)$ to $\mathbf{X}:=H^{3}(-1,1)\cap H_0^2(-1,1)$, and thus a compact map from $\mathbf{X}$ into itself.

	Clearly, in order to show the uniqueness of solutions for problem   \eqref{stream}-\eqref{streamBC}, i.e., the problem $\mathcal{L}\psi= \mathcal{L}_\sigma(\psi+\sigma \mathcal{L}_\sigma^{-1}\psi) =0$ has only zero solution,  we need only to show that the problem
	\begin{equation}\label{homopb}
		\psi+\sigma \mathcal{L}_\sigma^{-1}\psi =0
	\end{equation}
	has only zero solution. According to Fredholm alternative theorem,  we need only to show that for any $\phi \in \mathbf{X}$, the problem
	\begin{equation}\label{ontoeq}
		\psi + \sigma \mathcal{L}_\sigma^{-1} \, \psi = \phi
	\end{equation}
	always has a solution in  $\mathbf{X}$.
	For every function $\phi \in \mathbf{X}$, there exists a $\mfg \in H^{-1}(-1, 1)$, such that
	$ \mathcal{L}_\sigma \phi = \mfg, $ i.e., $\phi = \mathcal{L}_\sigma^{-1} (\mfg).$
	On the other hand, for any $\mfg\in H^{-1}(-1,1)$, there exist $\mfg_1$, $\mfg_2\in L^2(-1,1)$ such that $\mfg=\mfg_1'+\mfg_2$ (cf. \cite{Adams}).
	In fact, following Steps 1-5, for the functions  $\mfg_1$, $\mfg_2\in L^2(-1,1)$, one can also show that the problem
	\begin{equation*}
		\left\{ \ba
		&\mathcal{L} \psi = \mfg_1' + \mfg_2, \ \ \ \mbox{in}\ (-1, 1), \\
		&\psi(\pm 1) = \psi^{\prime}(\pm 1) = 0
		\ea
		\right.
	\end{equation*}
	has a solution $\psi \in\mathbf{X}$. Hence  $\psi$ satisfies
	\[
	\sigma \psi +\mathcal{L}_\sigma \psi =\mathcal{L}_\sigma\phi.
	\]
	Therefore,  for each $\phi\in \mathbf{X}$, there exists a  $\psi\in \mathbf{X}$ satisfies \eqref{ontoeq}.
	This gives the  uniqueness of problem \eqref{stream}-\eqref{streamBC} so that  the proof of Proposition \ref{mediumstream} is completed.
\end{proof}


\begin{pro} \label{mediumv}
	Assume that $\Phi\ge \tilde{C}(1+L)^{63}$ and $1 \le |n| \le \epsilon_1L\sqrt{\Phi}$, where $\tilde{C}$ and $\epsilon_1$ are the constants indicated in Proposition \ref{mediumstream}. The corresponding velocity field  $\Bv_n$  satisfies
	{ \be \label{est5-16}
		\|\Bv_n\|_{L^2(\Omega)} \leq C(1+L)  |\Phi \hat{n} |^{-\frac{1}{2}} \|\BF_n\|_{L^2(\Omega)} \leq C(1+L)^{\frac32}  \Phi^{-\frac12} \|\BF_n \|_{L^2(\Omega)}
		\leq C \Phi^{-\frac{10}{21}} \| \BF_n\|_{L^2(\Omega)},
		\ee}
	\be \label{est5-17}
	\|\Bv_n\|_{H^2(\Omega)} \leq C(1+L)^\frac56 |\Phi \hat{n} |^{\frac{1}{12}} \|\BF_n\|_{L^2(\Omega)} \leq C(1+L)^\frac56 \Phi^\frac18  \|\BF_n \|_{L^2(\Omega)},
	\ee
	and
	{ \be \label{est5-17-1}
		\|\Bv_n\|_{H^\frac53(\Omega)} \leq C_2  \|\BF_n \|_{L^2(\Omega)},
		\ee}
	where $C$ and $C_2$ are uniform constants independent of $L$, $n$, $\Phi$.
\end{pro}

\begin{proof}
	Let
	\begin{equation}\nonumber
		\omega_n=\left(\frac{d^2}{dy^2}- \hat{n}^2\right)\psi_n(y)e^{i\hat{n}x}
	\end{equation}
	be the corresponding vorticity. Straightforward computations show that $\Bv_n$ satisfies the elliptic equation
	\be \label{vorticity}
	\Delta \Bv_n = \left(-\frac{d}{dy}\omega_n, \ i \hat{n} \omega_n \right)e^{i\hat{n}x}\ \ \ \ \mbox{in}\ \Omega.
	\ee
	Applying the regularity theory for the elliptic equation with homogeneous boundary conditions (\hspace{1sp}\cite{GT}) gives
	\be \label{5-3-96}
	\| \Bv_n\|_{H^2(\Omega) } \leq C \left\|\nabla(\omega_n e^{i\hat{n}x})\right\|_{L^2(\Omega)}+C\| \Bv_n\|_{L^2(\Omega) } .
	\ee
	It follows from Proposition \ref{mediumstream} that one has
	\be \label{5-3-97}
	\begin{aligned}
		\|\Bv_n\|_{L^2(\Omega)}^2 \leq &C L\inte |\psi_n'|^2+\hat{n}^2|\psi_n|^2 \,dy
		\leq C(1+L)^2 |\Phi\hat{n}|^{-1}\|\BF_n\|_{L^2(\Omega)}^2
	\end{aligned}
	\ee
	and
	\be \label{5-3-98}
	\begin{aligned}
		\left\|\nabla(\omega_n e^{i\hat{n}x})\right\|_{L^2(\Omega)}^2 \leq&
		CL\inte \left|\psi_n^{(3)}\right|^2+\hat{n}^2|\psi_n''|^2+\hat{n}^4|\psi_n'|^2
		+\hat{n}^6|\psi_n|^2\,dy \\
		\leq& C(1+L)^{\frac53} |\Phi\hat{n}|^{\frac16}\|\BF_n\|^2_{L^2(\Omega)}.
	\end{aligned}
	\ee
	The estimate \eqref{5-3-97} is exactly \eqref{est5-16}. Substituting \eqref{5-3-97}-\eqref{5-3-98} into \eqref{5-3-96} yields \eqref{est5-17}. Finally, applying interpolation inequality gives
	\begin{equation*}
		\|\Bv_n\|_{H^{\frac53}(\Omega)} \leq \|\Bv_n\|_{L^2(\Omega)}^\frac{1}{6}\|\Bv_n\|_{H^2(\Omega)}^\frac{5}{6}\leq C(1+L)^\frac{31}{36} |\Phi\hat{n}|^{-\frac{1}{72}} \|\BF_n\|_{L^2(\Omega)} \leq  C_2 \|\BF_n\|_{L^2(\Omega)},
	\end{equation*}
	where the assumption $\Phi \geq \tilde{C} (1 + L)^{63}$ has been used. This finishes the proof of the proposition.
\end{proof}

\begin{remark}\label{uniformL-1}
	The constant $C$ appeared in \eqref{5-3-96} is independent of $L$. Indeed, the regularity theory for the elliptic equations gives that for any $k\in \mathbb{R}$, one has
	\begin{equation*}
		\|\Bv_n\|_{H^2(\Omega_{k-1,k})}^2 \leq C\left(\left\|\nabla(\omega_n e^{i\hat{n}x})\right\|_{L^2(\Omega_{k-\frac32,k+\frac12})}^2+\| \Bv_n\|_{L^2(\Omega_{k-\frac32,k+\frac12}) }^2\right),
	\end{equation*}
	where $\Omega_{a,b}=\{(x,y):~x\in (a,b),~y\in (-1,1)\}$ and $C$ is a uniform constant independent of $k$. Then for $L\ge 1$, we denote $N=[\pi L]$. Then one has
	\begin{equation}\label{largeL}
		\begin{aligned}
			&\|\Bv_n\|_{H^2(\Omega)}^2\leq \sum_{k=-N}^{N+1}\|\Bv_n\|_{H^2(\Omega_{k-1,k})}^2\\
			\leq& C\sum_{k=-N}^{N+1}\left(\left\|\nabla(\omega_n e^{i\hat{n}x})\right\|_{L^2(\Omega_{k-\frac32,k+\frac12})}^2+\| \Bv_n\|_{L^2(\Omega_{k-\frac32,k+\frac12}) }^2\right)\\
			\leq& 2C\left(\left\|\nabla(\omega_n e^{i\hat{n}x})\right\|_{L^2(\Omega_{-N-\frac32,N+\frac32})}^2+\| \Bv_n\|_{L^2(\Omega_{-N-\frac32,N+\frac32}) }^2\right)\\
			\leq &6C\left(\left\|\nabla(\omega_n e^{i\hat{n}x})\right\|_{L^2(\Omega)}^2+\| \Bv_n\|_{L^2(\Omega) }^2\right),
		\end{aligned}
	\end{equation}
	since $\Omega$ is periodic. When $0<L<1$, there exists an integer $M$ such $ML>1$. Noting that $\Omega$ can also be viewed as a periodic channel $\tilde{\Omega}$ with period $2\pi ML$, then it  follows from \eqref{largeL} that
	\begin{equation*}
		\begin{aligned}
			M\|\Bv_n\|_{H^2(\Omega)}^2=&\|\Bv_n\|_{H^2(\tilde{\Omega})}^2\leq 6C\left(\left\|\nabla(\omega_n e^{i\hat{n}x})\right\|_{L^2(\tilde{\Omega})}^2+\| \Bv_n\|_{L^2(\tilde{\Omega}) }^2\right)\\
			\leq&6MC\left(\left\|\nabla(\omega_n e^{i\hat{n}x})\right\|_{L^2(\Omega)}^2+\| \Bv_n\|_{L^2(\Omega) }^2\right).
		\end{aligned}
	\end{equation*}
	This, together with \eqref{largeL}, shows that the constant $C$ appeared in \eqref{5-3-96} is independent of $L$.
\end{remark}

\subsection{Uniform estimates for the case with large flux and high frequency}
In this subsection, we give the uniform estimates for the solutions of \eqref{stream}-\eqref{streamBC} with respect to the flux $\Phi$ when the flux is large and the frequency is high.

\begin{pro}\label{highstream}
	Assume that $\epsilon_1^2 \Phi \geq 276$ and $|n| \geq \epsilon_1 L \sqrt{\Phi}$ for some constant $\epsilon_1\in (0,1)$. Let $\psi_n$ be a smooth solution to the problem \eqref{stream}--\eqref{streamBC}, then one has
	\begin{equation}\label{est5-18}\begin{aligned}
			\inte \Phi|\hat{n}|(1-y^2)\left|\psi_n'\right|^2+\Phi|\hat{n}|^3(1-y^2)|\psi_n|^2\,dy&\\
			+\inte |\psi_n''|^2+\hat{n}^2\left|\psi_n'\right|^2 +\hat{n}^4|\psi_n|^2\,dy&\le C(\epsilon_1)|\hat{n}|^{-2}\inte |\BF_n|^2\,dy.
	\end{aligned}\end{equation}
	Here the constant $C(\epsilon_1)$ depends only on $\epsilon_1$.
\end{pro}

\begin{proof}
	Multiplying the equation \eqref{stream} by $\overline{\psi_n }$ and integrating the resulting equation over $[-1,1]$ yield
	\begin{equation}\label{5-3-21}
		\int_{-1}^1 \hat{n}^4|\psi_n |^2+2\hat{n}^2|\psi_n '|^2+|\psi''_n |^2\,d y=-\Re\int_{-1}^1 f_n\overline{\psi_n }\,dy+\Im\int_{-1}^1 \hat{n}U'\psi_n '\overline{\psi_n }\,dy
	\end{equation}
	and
	\begin{equation}\label{5-3-22}
		\frac{3\Phi\hat{n}}{4}\int_{-1}^1 \hat{n}^2|\psi_n |^2(1-y^2)+|\psi_n '|^2(1-y^2)\,d y=-\Im\int_{-1}^1 f_n\overline{\psi_n }\,dy+\frac{3\Phi\hat{n}}{4}\int_{-1}^1|\psi_n |^2\, dy.
	\end{equation}
	The equality \eqref{5-3-21} gives
	\begin{equation}\label{5-4-1}
		\inte |\psi_n''|^2+ \hat{n}^2\left|\psi_n'\right|^2+ \hat{n}^4|\psi_n|^2\,dy
		\le C\left| \int_{-1}^1 f_n \overline{\psi_n}  \, dy  \right| +C\Phi |\hat{n}|\inte \left|y\psi_n'\overline{\psi_n}\right|\,dy.
	\end{equation}
	Noting that $\hat{n}^2= \frac{n^2}{L^2}\ge \epsilon_1^2\Phi\geq 276$, one uses Lemma \ref{lemmaHLP} to obtain
	\begin{equation}\nonumber
		\int_{-1}^1|\psi_n |^2\, dy\leq \frac13 \int_{-1}^1 \hat{n}^2|\psi_n |^2(1-y^2)+|\psi_n '|^2(1-y^2)\,dy.
	\end{equation}
	This, together with \eqref{5-3-22}, yields
	\begin{equation}\label{5-4-2}
		\begin{aligned}
			\Phi |\hat{n}|\inte \left|\psi_n\right|^2+ \left|\psi_n'\right|^2(1-y^2)+\hat{n}^2 |\psi_n|^2(1-y^2)\,dy
			\le C\left| \int_{-1}^1 f_n \overline{\psi}_n \, dy  \right| .
	\end{aligned}\end{equation}
	It follows from	Lemma \ref{lemmaHLP} that
	\be\nonumber
	\ba
	& \Phi |\hat{n}| \int_{-1}^1 | y \psi_n' \bar{\psi}_n| \, dy
	\leq  C \Phi |\hat{n}| \left( \int_{-1}^1 |\psi_n'|^2 \, dy  \right)^{\frac12} \left( \int_{-1}^1 |\psi_n|^2 \, dy \right)^{\frac12} \\
	\leq & C \Phi |\hat{n}| \left[\left( \int_{-1}^1 |\psi_n''|^2 \, dy \right)^{\frac16} \left( \int_{-1}^1 (1 - y^2) |\psi_n'|^2 \, dy \right)^{\frac13}+\left( \int_{-1}^1 (1 - y^2) |\psi_n'|^2 \, dy \right)^{\frac12}\right]\\
	&\ \ \ \ \ \ \ \times \left[\left( \int_{-1}^1 |\psi_n'|^2 \, dy \right)^{\frac16} \left(  \int_{-1}^1 (1 - y^2) |\psi_n|^2 \, dy \right)^{\frac13} +\left( \int_{-1}^1 (1 - y^2) |\psi_n|^2 \, dy \right)^{\frac12}\right].
	\ea\ee
	Using \eqref{5-4-2} and H\"{o}lder inequality gives
	\be\label{5-4-3}
	\ba
	& \Phi |\hat{n}| \int_{-1}^1 | y \psi_n' \bar{\psi}_n| \, dy \\
	\leq & C \Phi |\hat{n}| \left[\left( \int_{-1}^1 |\psi_n''|^2 \, dy \right)^{\frac16} \left( \Phi|\hat{n}|\right)^{-\frac13}\left| \int_{-1}^1 f_n \overline{\psi}_n \, dy  \right|^\frac13+\left( \Phi|\hat{n}|\right)^{-\frac12}\left| \int_{-1}^1 f_n \overline{\psi}_n \, dy  \right|^\frac12\right]\\
	&\ \ \ \ \ \ \ \times \left[\left( \int_{-1}^1 \hat{n}^2 |\psi_n'|^2 \, dy \right)^{\frac16}\left( \Phi|\hat{n}|^4\right)^{-\frac13}\left| \int_{-1}^1 f_n \overline{\psi}_n \, dy  \right|^\frac13+\left( \Phi|\hat{n}|^3\right)^{-\frac12}\left| \int_{-1}^1 f_n \overline{\psi}_n \, dy  \right|^\frac12\right].
	\ea\ee
	Using H\"{o}lder inequality gives
	\be
	\ba
	& \Phi |\hat{n}| \int_{-1}^1 | y \psi_n' \bar{\psi}_n| \, dy \\
	\leq & \frac14 \int_{-1}^1 |\psi_n''|^2 \, dy + \frac14 \int_{-1}^1 \hat{n}^2 |\psi_n'|^2 \, dy +
	C (\Phi^{\frac12} |\hat{n}|^{-1}+\Phi^{\frac15} |\hat{n}|^{-1}+|\hat{n}|^{-1}) \left|\int_{-1}^1 f_n \bar{\psi}_n \, dy  \right|\\
	\leq & \frac14 \int_{-1}^1 |\psi_n''|^2 \, dy + \frac14 \int_{-1}^1 \hat{n}^2 |\psi_n'|^2 \, dy +
	C \left|\int_{-1}^1 f_n \bar{\psi}_n \, dy  \right|.
	\ea
	\ee
	Taking \eqref{5-4-3} into \eqref{5-4-1} yields
	\be \nonumber
	\ba
	& \inte |\psi_n''|^2+ \hat{n}^2\left|\psi_n'\right|^2+ \hat{n}^4|\psi_n|^2\,dy\le  C\left|\int_{-1}^1 f_n \bar{\psi}_n \, dy  \right| \\
	\leq & C|\hat{n}|^{-1}  \left( \int_{-1}^1 |\BF_n|^2 \, dy  \right)^{\frac12} \left( \int_{-1}^1 \hat{n}^4 |\psi_n|^2 + \hat{n}^2|\psi_n'|^2 \, dy \right)^{\frac12}.
	\ea
	\ee
	By Young's inequality, one has
	\be \label{5-4-5}
	\inte |\psi_n''|^2+ \hat{n}^2\left|\psi_n'\right|^2+ \hat{n}^4|\psi_n|^2\,dy
	\leq C |\hat{n}|^{-2} \int_{-1}^1 |\BF_n|^2 \, dy .
	\ee
	Substituting \eqref{5-4-5} into \eqref{5-4-2} gives the inequality \eqref{est5-18}. This finishes the proof of the proposition.
\end{proof}

Following the proof in \cite[Appendix B]{SWX},  a priori estimates established in Proposition \ref{highstream} together with the Galerkin method give the existence of  solution for the problem \eqref{stream}-\eqref{streamBC} for any $|\hat{n}|\ge \epsilon_1 \sqrt{\Phi}\ge \sqrt{276}$.

\begin{pro}\label{highv}
	Assume that $\epsilon_1^2 \Phi \geq 276$ and  $|n|\ge \epsilon_1 L\sqrt{\Phi}$ for some constant $\epsilon_1\in (0,1)$. The corresponding velocity $\Bv_n$ satisfies
	\begin{equation}\nonumber
		\|\Bv_n\|_{L^2(\Omega)}\le C\Phi^{-1}\|\BF_n\|_{L^2(\Omega)},
	\end{equation}
	\begin{equation}\label{est5-21}
		\|\Bv_n\|_{H^2(\Omega)}\le C(1+\Phi^\frac14)\|\BF_n\|_{L^2(\Omega)},
	\end{equation}
	and
	\begin{equation}\nonumber
		\|\Bv_n\|_{H^\frac53(\Omega)}\le C_3\|\BF_n\|_{L^2(\Omega)},
	\end{equation}
	for some positive constants $C$ and $C_3$ independent of flux $\Phi$, $n,L$, and $\BF_n$.
\end{pro}
\begin{proof}
	First, by virtue of \eqref{est5-18}, one has
	\begin{equation}\nonumber
		\begin{aligned}
			\|\Bv_n\|_{L^2(\Omega)}\le&CL^\frac12\left(\inte \hat{n}^2|\psi_n|^2+\left|\psi_n'\right|^2 \,dy\right)^\frac12\\
			\le&C\hat{n}^{-2}L^\frac12\left(\inte |\BF_n|^2\,dy\right)^\frac12\le C\Phi^{-1}\|\BF_n\|_{L^2(\Omega)}.
		\end{aligned}
	\end{equation}
	Since $\psi_n$ is a solution to the problem \eqref{stream}-\eqref{streamBC}, direct computations yield
	\be \nonumber
	{\rm curl}~\left( (\boldsymbol{U}\cdot \nabla ) \Bv_n + (\Bv_n \cdot \nabla) \boldsymbol{U}  \right) - {\rm curl}~(\Delta \Bv_n) = {\rm curl}~\BF_n.
	\ee
	Therefore, for each $n$, there exists some function $P_n$ with $\nabla P_n\in L^2(\Omega)$ such that
	\be \nonumber
	(\boldsymbol{U}\cdot \nabla ) \Bv_n + (\Bv_n \cdot \nabla) \boldsymbol{U} - \Delta \Bv_n + \nabla P_n = \BF_n.
	\ee
	Then $\Bv_n$ satisfies the equation
	\be \label{stokes}
	\left\{  \ba
	& -\Delta \Bv_n+ \nabla P_n = \BF_n-\boldsymbol{U}\partial_x  \Bv_n + v_{2,n}e^{i\hat{n} x}  \boldsymbol{U}' , \ \ \ \mbox{in}\ \ \Omega, \\
	& {\rm div}~\Bv_n = 0,\ \ \ \ \ \ \ \ \ \ \ \ \ \ \ \ \ \ \ \ \ \ \ \ \ \ \ \ \ \ \ \ \ \ \ \ \ \ \ \ \ \ \ \ \ \ \mbox{in} \ \ \Omega.
	\ea  \right.
	\ee
	According to the regularity theory for Stokes equations (\hspace{1sp}\cite[Lemma VI.1.2]{Galdi}), one has
	\begin{equation}\label{5-4-8}
		\begin{aligned}
			\|\Bv_n \|_{H^2(\Omega)}\le& C(\|\BF_n\|_{L^2(\Omega)} +\Phi \|(1 - y^2) \partial_x \Bv_n\|_{L^2(\Omega)} + \Phi \|v_{2,n}e^{i\hat{n} x}\|_{L^2(\Omega)} +  \|\Bv_n\|_{H^1(\Omega)}).
	\end{aligned}\end{equation}
	Similar to Remark \ref{uniformL-1}, one can also show that the constant $C$ appeared in \eqref{5-4-8} is independent of $L$.
	
	It follows from \eqref{est5-18} that one has
	\begin{equation}\label{5-4-9}
		\begin{aligned}
			&\Phi\|(1-y^2)\partial_x \Bv_n \|_{L^2(\Omega)}\\
			\le& C\Phi L^\frac12\left(\inte \hat{n}^4(1-y^2)^2|\psi_n|^2+\hat{n}^2(1-y^2)^2\left|\psi_n'\right|^2\,dy\right)^\frac12\\
			\le&C|\Phi\hat{n}|^\frac12L^\frac12\left(\inte \Phi|\hat{n}|^3(1-y^2)|\psi_n|^2+\Phi|\hat{n}|(1-y^2)\left|\psi_n'\right|^2\,dy\right)^\frac12\\
			\le&C\Phi^\frac12|\hat{n}|^{-\frac12}L^\frac12\left(\inte |\BF_n|^2\,dy\right)^\frac12\le C\Phi^\frac14\|\BF_n\|_{L^2(\Omega)}
	\end{aligned}\end{equation}
	and
	\begin{equation}\label{5-4-10}
		\begin{aligned}
			\Phi\|v_{2,n}e^{i\hat{n}x}\|_{L^2(\Omega)}\le& C\Phi\|\Bv_n\|_{L^2(\Omega)}\le C\|\BF_n\|_{L^2(\Omega)}.
	\end{aligned}\end{equation}
	By Poincar\'{e}'s inequality and Proposition \ref{highstream}, one has
	\begin{equation}\label{5-4-11}
		\begin{aligned}
			\|\Bv_n \|_{H^1(\Omega)}\le C\|\nabla\Bv_n \|_{L^2(\Omega)}
			\le & CL^\frac12\left(\inte \hat{n}^4|\psi_n|^2+\hat{n}^2\left|\psi_n'\right|^2 +|\psi_n''|^2\,dy\right)^\frac12\\
			\le & C \Phi^{-\frac12}\|\BF_n\|_{L^2(\Omega)}.
	\end{aligned}\end{equation}
	Combining \eqref{5-4-8}-\eqref{5-4-11} yields \eqref{est5-21}. Finally, using interpolation between $H^2(\Omega)$ and $H^1(\Omega)$ gives
	\begin{equation}\nonumber
		\|\Bv_n \|_{H^\frac53(\Omega)}\le C\|\Bv_n \|_{H^1(\Omega)}^\frac13\|\Bv_n \|_{H^2(\Omega)}^\frac23\le C_3\|\BF_n \|_{L^2(\Omega)}.
	\end{equation}
	Hence the proof of the proposition is completed.
\end{proof}

As shown in the proof of Proposition \ref{highv}, for each $n$, there exists a function $P_n$ with $\nabla P_n\in L^2(\Omega)$ such that
\be \nonumber
(\boldsymbol{U}\cdot \nabla ) \Bv_n + (\Bv_n \cdot \nabla) \boldsymbol{U} - \Delta \Bv_n + \nabla P_n = \BF_n.
\ee
According to Propositions \ref{0mode}, \ref{mediumv}, and \ref{highv},  for $\Phi \geq \tilde{C} ( 1 + L)^{63}$, one has
\be \nonumber
\|\Bv\|_{H^2(\Omega)} \leq C\max \{ (1+ L)^{\frac56} |\Phi|^{\frac{1}{8}} , \, \Phi^{\frac14}   \} \|\BF\|_{L^2(\Omega)}\leq C \Phi^{\frac14}
\| \BF\|_{L^2(\Omega)}.
\ee
This finishes the proof of Theorem \ref{thm1}.


\section{The nonlinear problem}\label{sec-nonlinear}
In this section, we prove the existence and uniqueness of  solution to the nonlinear problem \eqref{model11}-\eqref{model11'}.
Let $\Bv$ be a periodic solution of the nonlinear system \eqref{model11}-\eqref{model11'}. Then for any $n\in \mathbb{Z}$, the stream function $\psi_n$ of $\Bv_n$ satisfies the following nonlinear system
\begin{equation}\label{model71}
	\begin{aligned}
		&-i\hat{n}U''\psi_n+i\hat{n}U\left(\frac{d^2}{dy^2}-\hat{n}^2\right)\psi_n-\left(\frac{d^2}{dy^2}-\hat{n}^2\right)^2\psi_n\\
		=&i\hat{n}F_{2,n}-\frac{d}{dy}F_{1,n}-\frac{d}{dy}\left(\sum\limits_{m\in \mathbb{Z}} v_{2,n-m}\omega_{m}\right)-i\hat{n}\sum\limits_{m\in \mathbb{Z}} v_{1,n-m}\omega_{m},
	\end{aligned}
\end{equation}
and the boundary conditions \eqref{streamBC}, where $\omega_n$ is defined in terms of $\psi_n$ as that in \eqref{defomegan}.

\begin{pro}
	\label{largeF}
	Assume that $\BF\in L^2(\Omega)$. For every $\Phi\ge \tilde{C}(1+L)^{63 }$, if
	\begin{equation}\nonumber
		\|\BF\|_{L^2(\Omega)}\le \Phi^\frac{1}{32},
	\end{equation}
	the nonlinear problem \eqref{model11}-\eqref{model11'} admits a unique solution $\Bv\in H^2(\Omega)$ satisfying
	{ \begin{equation}\nonumber
			\left(\left\|\Bv_0\right\|_{H^2(\Omega)}^2+ \sum\limits_{n\neq0}\left\|\Bv_n\right\|_{H^\frac53(\Omega)}^2\right)^\frac12
			\le C\|\BF\|_{L^2(\Omega)}
		\end{equation}
		and
		\begin{equation}\nonumber
			\|v_2\|_{L^2(\Omega)}\leq \Phi^{-\frac{5}{12}}, \quad
			\left\|\Bv\right\|_{H^2(\Omega)}\le C \Phi^\frac14\|\BF\|_{L^2(\Omega)},
	\end{equation}}
	where $C$ is a uniform constant.
\end{pro}

\begin{proof}
	The proof is divided into five steps. First, we consider the case that $L \geq 1$. The existence of the solutions is established by the iteration method.
	
	{\em Step 1: Iteration scheme.}  Given $\BF\in L^2(\Omega)$, the linear problem \eqref{stream}-\eqref{streamBC} admits a unique solution $\psi_n^0$ for each $n$. The corresponding velocity field is denoted by
	\begin{equation}\nonumber
		\Bv^0=\sum\limits_{n\in\Z}v_{1,n}^0e^{i\hat{n}x}\Be_1+v_{2,n}^0e^{i\hat{n}x}\Be_2  \text{ with }
		v^0_{2,n}=i\hat{n}\psi_n^0~~\text{ and }v^0_{1,n}=-\frac{d}{dy}\psi_n^0.
	\end{equation}
	For each $j\ge 0$, let $\psi^{j+1}_n$ be the solution of the iteration problem
	\begin{equation}\label{model72}
		\left\{\begin{aligned}
			-i\hat{n}U''\psi_n^{j+1}&+i\hat{n}U\left(\frac{d^2}{dy^2}-\hat{n}^2\right)\psi^{j+1}_n-\left(\frac{d^2}{dy^2}-\hat{n}^2\right)^2\psi^{j+1}_n\\
			=&i\hat{n} (F_{2,n}+F^j_{2,n})-\frac{d}{dy}(F_{1,n}+F_{1,n}^j),\\
			\psi_n^{j+1}(\pm1)=&\frac{d}{dy}\psi_n^{j+1}(\pm1)=0,
		\end{aligned}\right.
	\end{equation}
	where
	\begin{equation}\nonumber
		\begin{aligned}
			F^j_{2,n}=&-\sum\limits_{m\in \mathbb{Z}} v_{1,n-m}^j \omega_{m}^j ,~~\ F^j_{1,n}=\sum\limits_{m\in \mathbb{Z}} v_{2,n-m}^j \omega_{m}^j, \\
		\end{aligned}
	\end{equation}
	and
	\begin{equation}\nonumber
		v^j_{2,n}=i\hat{n}\psi_n^j,~~v^j_{1,n}=-\frac{d}{dy}\psi_n^j,~~\omega^j_{n}=\left(\frac{d^2}{dy^2}-\hat{n}^2\right)\psi_n^j.
	\end{equation}
	For convenience, we define the projection operator
	\begin{equation}\label{projection}
		\Q \Bv=\sum\limits_{n\neq 0}(v_{1,n}\Be_1+v_{2,n}\Be_2)e^{i\hat{n}x}.
	\end{equation}
	
	Assume that  $\|\BF\|_{L^2(\Omega)}\le \Phi^\frac{1}{32}$ and set
	\begin{equation}\nonumber
		\mathcal{J}=\left\{\Bv=\sum\limits_{n\in\Z} \Bv_n
		\left|\begin{array}{l}v_{2,0}=0,~~\ \ \left\|\Q\Bv\right\|_{L^2(\Omega)}\le  \Phi^{-\frac{5}{12}},\\ \left(\left\|\Bv_0\right\|_{H^2(\Omega)}^2+ \left\|\Q\Bv\right\|_{H^\frac53(\Omega)}^2\right)^\frac12\le 2C_4\|\BF\|_{L^2(\Omega)}
		\end{array}\right.
		\right\},
	\end{equation}
	where $C_4=\max\{C_1,C_2,C_3\}$ and $C_i(i=1,2,3)$ are positive constants appeared in Propositions \ref{0mode}, \ref{mediumv}, and \ref{highv}, respectively.
	
	{\em Step 2: Mathematical induction.} According to the linear estimates obtained in Propositions \ref{0mode}, \ref{mediumv}, and \ref{highv}, one has $\Bv^0\in \mathcal{J}$ immediately, when $\tilde{C}$ is sufficiently large. Assume that $\Bv^j\in \mathcal{J}$, we claim that $\Bv^{j+1}\in \mathcal{J}$. First, it follows from Propositions \ref{0mode}, \ref{mediumv}, and \ref{highv} that
	\begin{equation}\begin{aligned}
			&\left(\left\|\Bv_0^{j+1}\right\|_{H^2(\Omega)}^2+  \left\|\Q\Bv^{j+1}\right\|_{ H^\frac53(\Omega)}^2 \right)^\frac12\\
			\le& C_4\left(\left\|F_{1,0}+F_{1,0}^j\right\|_{L^2(\Omega)}^2+\sum\limits_{n\neq 0}\left\|F_{2,n}+F_{2,n}^j\right\|_{L^2(\Omega)}^2+\left\|F_{1,n}+F_{1,n}^j\right\|_{L^2(\Omega)}^2\right)^\frac12\\
			\le& \sqrt{2}C_4\left\|\BF\right\|_{L^2(\Omega)}
			+C\left(\sum\limits_{n\in \mathbb{Z} }\left\|F^{j}_{1,n}\right\|_{L^2(\Omega)}^2+\sum\limits_{n\neq 0}\left\|F^{j}_{2,n}\right\|_{L^2(\Omega)}^2\right)^\frac12.
		\end{aligned}
	\end{equation}
	This, together with the assumption $v^j_{2,0}=0$, gives
	\begin{equation}\begin{aligned}\label{eq706}
			&\left(\left\|\Bv_0^{j+1}\right\|_{H^2(\Omega)}^2+  \left\|\Q\Bv^{j+1}\right\|_{ H^\frac53(\Omega)}^2 \right)^\frac12\\
			\le&\sqrt{2}C_4\left\|\BF\right\|_{L^2(\Omega)}+C\left[\sum\limits_{n\in \mathbb{Z} } \left\| \sum\limits_{m\neq 0,n}v^j_{2,n-m}\omega^j_{m}\right\|_{L^2(\Omega)}^2 + \sum\limits_{n \neq 0} \left\|\sum\limits_{m\neq 0,n}v^j_{1,n-m}\omega^j_{m}\right\|_{L^2(\Omega)}^2 \right.\\
			&\left.+\sum\limits_{n\neq 0}\left(\left\|v^j_{2,n}\omega^j_{0}\right\|_{L^2(\Omega)}^2+\left\|v^j_{1,n}\omega^j_{0}\right\|_{L^2(\Omega)}^2+\left\|v^j_{1,0}\omega^j_{n}\right\|_{L^2(\Omega)}^2\right)\right]^\frac12.
	\end{aligned}\end{equation}
	Using Hausdorff-Young inequality and the interpolation inequalities yields
	{\begin{equation}\begin{aligned}\label{eq707}
				&\left(\sum\limits_{n\in \mathbb{Z} }\left\| \sum\limits_{m\neq 0,n}v^j_{2,n-m}\omega^j_{m}\right\|_{L^2(\Omega)}^2\right)^\frac12=
				\left(\sum\limits_{n\in \mathbb{Z} } \left\| \sum\limits_{m\neq 0,n}(\Q v_2^j)_{n-m}(\Q \omega^j)_m\right\|_{L^2(\Omega)}^2\right)^\frac12\\
				\leq & \left(\sum\limits_{n\in \mathbb{Z} }\left\|(\Q v_2^j\cdot\Q \omega^j)_n\right\|_{L^2(\Omega)}^2\right)^\frac12=\left\|\Q v_2^j\cdot\Q \omega^j\right\|_{L^2(\Omega)} \le\left\|\Q v_2^j\right\|_{L^{6} (\Omega)}\left\|\Q \omega^j\right\|_{L^3(\Omega)}\\
				\le& C\left\|\Q \Bv^j\right\|_{L^2(\Omega)}^\frac25\left\|\Q \Bv^j\right\|_{H^\frac53(\Omega)}^\frac35\left\|\Q \Bv^{j}\right\|_{H^\frac53(\Omega)}\\
				\le& C \Phi^{-\frac16}\|\BF\|_{L^2(\Omega)}^\frac85.
	\end{aligned}\end{equation}}
	
	Similarly, it also holds that
	{\begin{equation}\label{eq708}
			\left(\sum\limits_{n\neq 0}\left\| \sum\limits_{m\neq 0,n}v^j_{1,n-m}\omega^j_{m}\right\|_{L^2(\Omega)}^2\right)^\frac12\leq \left\|\Q v_1^j\right\|_{L^6 (\Omega)}\left\|\Q \omega^j\right\|_{L^3(\Omega)}
			\le C \Phi^{-\frac16 } \|\BF\|_{L^2(\Omega)}^\frac85.
	\end{equation} }
	{Using H\"{o}lder inequality and Sobolev's embedding inequalities yields
		\begin{equation}\begin{aligned}\label{eq709}
				&\left(\sum\limits_{n\neq 0}\left\|v^j_{2,n}\omega^j_{0}\right\|_{L^2(\Omega)}^2+\sum\limits_{n\neq 0}\left\|v^j_{1,n}\omega^j_{0}\right\|_{L^2(\Omega)}^2\right)^\frac12\le C\left(\sum\limits_{n\neq 0}\left\|\Bv^{j}_n\right\|_{L^6 (\Omega)}^2\left\|\omega^j_{0}\right\|_{L^3(\Omega)}^2\right)^\frac12\\
				\le& C\left\|\Bv^{j}_{0}\right\|_{H^2(\Omega)} \left\|\Q\Bv^{j}\right\|_{L^2(\Omega)}^\frac25 \left\|\Q\Bv^{j}\right\|_{H^\frac53(\Omega)}^\frac35\\
				\le& C \Phi^{-\frac16 }\|\BF\|_{L^2(\Omega)}^\frac85
		\end{aligned}\end{equation}
		and
		\begin{equation}\begin{aligned}\label{eq710}
				&\left(\sum\limits_{n\neq 0}\left\|v^j_{1,0}\omega^j_{n}\right\|_{L^2(\Omega)}^2\right)^\frac12
				\le  C\|\Bv_{0}^{j}\|_{L^\infty(\Omega)}\|\Q\Bv^{j}\|_{H^1(\Omega)}\\
				\le& C\|\Bv_{0}^{j}\|_{H^2(\Omega)}\|\Q\Bv^{j}\|_{L^2(\Omega)}^\frac25\|\Q\Bv^{j}\|_{H^\frac53(\Omega)}^\frac35\le C \Phi^{-\frac16}\|\BF\|_{L^2(\Omega)}^\frac85.
		\end{aligned}\end{equation}
		Combining the estimates \eqref{eq706}-\eqref{eq710}, together with the assumption $\|\BF\|_{L^2(\Omega)}\le \Phi^\frac{1}{32}$,  gives
		\begin{equation}\nonumber
			\begin{aligned}
				&\left(\left\|\Bv_0^{j+1}\right\|_{H^2(\Omega)}^2+ \left\|\Q\Bv^{j+1}\right\|_{H^\frac53(\Omega)}^2\right)^\frac12
				\le \sqrt{2}C_4\|\BF\|_{L^2(\Omega)}+C \Phi^{-\frac{71}{480}}\|\BF\|_{L^2(\Omega)}.
		\end{aligned}\end{equation}
		Moreover, the estimates \eqref{eq706}-\eqref{eq710} together with  Propositions \ref{mediumv} and \ref{highv} yield that
		\begin{equation}\nonumber
			\begin{aligned}
				&\left\|\Q\Bv^{j+1}\right\|_{L^2(\Omega)}\le C\Phi^{-\frac{10}{21}}\left(\sum\limits_{n\neq0}\left\|F_{2,n}+F_{2,n}^j\right\|_{L^2(\Omega)}^2+\left\|F_{1,n}+F_{1,n}^j\right\|_{L^2(\Omega)}^2\right)^\frac12\\
				\le& C \Phi^{-\frac{10}{21}} \|\BF\|_{L^2(\Omega)}+ C \Phi^{-\frac{10}{21}} \Phi^{-\frac{71}{480}}\|\BF\|_{L^2(\Omega)}.
		\end{aligned}\end{equation}
		Then for any $\Phi\ge \tilde{C}(1+L)^{63}$, the solution $\Bv^{j+1}$ satisfies
		\begin{equation}\label{eq712}
			\begin{aligned}
				\left(\left\|\Bv_0^{j+1}\right\|_{H^2(\Omega)}^2+ \left\|\Q\Bv^{j+1}\right\|_{H^\frac53(\Omega)}^2\right)^\frac12\le2C_4\|\BF\|_{L^2{(\Omega)}}
		\end{aligned}\end{equation}
		and
		\begin{equation}\nonumber\begin{aligned}
				&\left\|\Q\Bv^ {j+1}\right\|_{L^2(\Omega)}\le \Phi^{-\frac{5}{12}},
		\end{aligned}\end{equation}
		provided $\tilde{C}$ is sufficiently large. Hence $\Bv^{j+1}\in \mathcal{J}$. By mathematical induction, 	$\Bv^j \in \mathcal{J}$ for each $j\in \mathbb{N}$ and $\{\Bv^j\}_{j\ge0}$ is a bounded sequence in
		\begin{equation}\nonumber
			\mathcal{J}_0=\left\{\Bv=\sum\limits_{n\in\Z} \Bv_n
			\left|\left(\left\|\Bv_0\right\|_{H^2(\Omega)}^2+ \left\|\Q\Bv\right\|_{H^\frac53(\Omega)}^2\right)^\frac12<\infty
			\right.\right\}.
		\end{equation}
		Therefore, there exists a function $\Bv\in \mathcal{J}_0$ such that $\Bv^j\rightharpoonup \Bv$ in $\mathcal{J}_0$ and
		\begin{equation}\label{eq715}
			\left(\left\|\Bv_0\right\|_{H^2(\Omega)}^2+   \left\|\Q\Bv\right\|_{H^\frac53(\Omega)}^2\right)^\frac12\le 2C_4\|\BF\|_{L^2(\Omega)}, \ \ \ \ \
			\|\mathcal{Q}\Bv\|_{L^2(\Omega)} \leq \Phi^{-\frac{5}{12}}.
	\end{equation} }
	Since $\psi^{j+1}_n$ is the solution to the problem \eqref{model72}, $\Bv^{j+1}$ satisfies
	\begin{equation}\nonumber
		\mathrm{curl}(-\Delta \Bv^{j+1}+\BU\cdot \nabla \Bv^{j+1}+\Bv^{j+1}\cdot \nabla \BU+\Bv^{j}\cdot\nabla\Bv^{j}-\BF)=0.
	\end{equation}
	Taking the limit for $j$ in the above equation yields
	\begin{equation}\nonumber
		\mathrm{curl}(-\Delta \Bv+\BU\cdot \nabla \Bv+\Bv\cdot \nabla \BU+\Bv\cdot\nabla\Bv-\BF)=0.
	\end{equation}
	Therefore, there exists a function $P$ with $\nabla P\in L^2(\Omega)$ such that
	\begin{equation}\label{eq718}
		-\Delta \Bv+\BU\cdot \nabla \Bv+\Bv\cdot \nabla \BU+\nabla P=-\Bv\cdot\nabla\Bv+\BF.
	\end{equation}
	
	{\em Step 3: $H^2$-regularity. }
	It follows from Propositions \ref{mediumv} and \ref{highv} that
	\begin{equation}\begin{aligned}\label{eq723}
			\left\|\Q\Bv\right\|_{H^2(\Omega)}
			\le&C  \Phi^\frac14\left[\left\|\Q\BF\right\|_{L^2(\Omega)}+\left(\sum\limits_{n\in \mathbb{Z} }\left\| \sum\limits_{m\neq 0,n}v_{2,n-m}\omega_{m}\right\|_{L^2(\Omega)}^2\right)^\frac12\right.\\
			&+\left(\sum\limits_{n\neq 0}\left\|\sum\limits_{m\neq 0,n}v_{1,n-m}\omega_{m}\right\|_{L^2(\Omega)}^2\right)^\frac12+\left(\sum\limits_{n\neq 0}\left\|v_{2,n}\omega_{0}\right\|_{L^2(\Omega)}^2\right)^\frac12\\
			&\left.+\left(\sum\limits_{n\neq 0}\left\|v_{1,n}\omega_{0}\right\|_{L^2(\Omega)}^2\right)^\frac12
			+\left(\sum\limits_{n\neq 0}\left\|v_{1,0}\omega_{n}\right\|_{L^2(\Omega)}^2\right)^\frac12\right].
	\end{aligned}\end{equation}
	With the aid of the estimates \eqref{eq715}, one can estimate the terms on the right-hand side of \eqref{eq723} as in Step 2. Hence it holds that
	{\begin{equation}\nonumber
			\left\|\Q\Bv\right\|_{H^2(\Omega)}
			\leq C \Phi^{\frac14} \left\|\BF\right\|_{L^2(\Omega)}+ C \Phi^{\frac14} \Phi^{- \frac{71}{480}} \left\|\BF\right\|_{L^2(\Omega)}.
	\end{equation}}
	This, together with \eqref{eq715}, implies that $\Bv$ is a strong solution of \eqref{eq718} and satisfies
	\begin{equation}\nonumber
		\left\|\Bv\right\|_{H^2(\Omega)}\le C \Phi^{\frac14}\left\|\BF\right\|_{L^2(\Omega)}.
	\end{equation}
	
	{\em Step 4: Uniqueness.} To prove the uniqueness of the solution $\Bv$, we assume that $\Bv^j\in \mathcal{J}(j=1,2)$ are two solutions of the nonlinear problem \eqref{model11}-\eqref{model11'} satisfying
	\begin{equation}\nonumber
		\left(\left\|\Bv^{j}_0\right\|_{H^2(\Omega)}^2+  \left\|\Q\Bv^{j}\right\|_{H^\frac53(\Omega)}^2\right)^\frac12\le C\left\|\BF\right\|_{L^2(\Omega)}\le C\Phi^\frac{1}{32}, \ \ \ \|\Q \Bv^j\|_{L^2(\Omega)}\leq \Phi^{-\frac{5}{12}},\,\,\text{for } j=1,2.
	\end{equation}
	In addition, for each $n$, the difference of the stream functions $\tilde{\psi}_n:=\psi_n^1-\psi_n^2$ satisfies
	\begin{equation}\nonumber
		\begin{aligned}
			&-i\hat{n}U''\tilde{\psi}_n+i\hat{n}U\left(\frac{d^2}{dy^2}-\hat{n}^2\right)\tilde{\psi}_n-\left(\frac{d^2}{dy^2}-\hat{n}^2\right)^2\tilde{\psi}_n\\
			=&i\hat{n}(F^1_{2,n}-F^2_{2,n})
			-\frac{d}{dy}(F^1_{1,n}-F^2_{1,n}),
		\end{aligned}
	\end{equation}
	where
	\[
	F_{1,n}^i = \sum\limits_{m\in \mathbb{Z}}v_{2,n-m}^i\omega_{m}^i\quad \text{and}\quad F_{2,n}^i=\sum\limits_{m\in \mathbb{Z}}v_{1,n-m}^i\omega_{m}^i.
	\]
	Denote
	\[
	\tilde{\Bv}=\Bv^1-\Bv^2\quad \text{and}\quad \tilde{\omega}=\omega^1-\omega^2.
	\]
	First, it follows from Proposition \ref{0mode} and the Sobolev's embedding inequalities that
	{ \begin{equation}\begin{aligned}\label{eq724}
				&\|\tilde{\Bv}_0\|_{H^2(\Omega)}\le C\left\|F^1_{1,0}-F^2_{1,0}\right\|_{L^2(\Omega)}
				\le C\left\|\sum\limits_{m\neq 0}v_{2,-m}^1\omega_{m}^1-v_{2,-m}^2\omega_{m}^2\right\|_{L^2(\Omega)}\\
				\le&C\left\|\sum\limits_{m\neq 0}v_{2,-m}^1\left(\omega_{m}^1-\omega_{m}^2\right)\right\|_{L^2(\Omega)}+C\left\|\sum\limits_{m\neq 0}\left(v_{2,-m}^1-v_{2,-m}^2\right)\omega_{m}^2\right\|_{L^2(\Omega)}\\
				\le& C\left\|\left(\Q v_{2}^1\cdot \Q(\omega^1-\omega^2)\right)_0\right\|_{L^2(\Omega)}+C\left\|\left(\Q(v_{2}^1-v_{2}^2)\cdot\Q\omega^2\right)_0\right\|_{L^2(\Omega)}\\
				\le&C\left\|\Q\Bv^{1}\right\|_{L^6(\Omega)}\left\|\Q\tilde{\omega}\right\|_{L^3(\Omega)}
				+C\left\|\Q\tilde{\Bv}\right\|_{L^\infty(\Omega)}\left\|\Q\Bv^{2}\right\|_{H^1(\Omega)}\\
				\le&C\left\|\Q\Bv^{1}\right\|_{L^2(\Omega)}^\frac{2}{5} \left\|\Q\Bv^{1}\right\|_{H^\frac53(\Omega)}^\frac{3}{5}
				\left\|\Q\tilde{\Bv}\right\|_{H^\frac32(\Omega)}
				\\
				&+C\left\|\Q\tilde{\Bv}\right\|_{H^\frac32(\Omega)}\left\|\Q\Bv^{2}\right\|_{L^2(\Omega)}^\frac25 \left\|\Q\Bv^{2}\right\|_{H^\frac53(\Omega)}^\frac35\\
				\le&C  \Phi^{-\frac{71}{480}}\left\|\Q\tilde{\Bv}\right\|_{H^\frac32(\Omega)}.
		\end{aligned}\end{equation}
		In addition, applying Propositions \ref{mediumv} and \ref{highv}, and the interpolation inequalities yields
		\begin{equation}\begin{aligned}\label{eq725}
				&\|\Q\tilde{\Bv}\|_{H^\frac32(\Omega)}\le C\|\Q\tilde{\Bv}\|_{L^2(\Omega)}^\frac{1}{10}\|\Q\tilde{\Bv}\|_{H^\frac53(\Omega)}^\frac{9}{10}\\
				\le& C  \Phi^{-\frac{1}{21}} \left(\sum\limits_{n\neq 0}\left\|F^{1}_{2,n}-F^{2}_{2,n}\right\|_{L^2(\Omega)}^2
				+\left\|F^{1}_{1,n}-F^{2}_{1,n}\right\|_{L^2(\Omega)}^2\right)^\frac12.
		\end{aligned}\end{equation}
		It follows from the Hausdorff-Young inequality and Sobolev's embedding inequalities that one has
		\begin{equation}\begin{aligned}\label{eq726}
				&\left(\sum\limits_{n\neq0}\left\| \sum\limits_{m\neq 0,n}v^1_{2,n-m}\omega^1_{m}-v^2_{2,n-m}\omega^2_{m}\right\|_{L^2(\Omega)}^2\right)^\frac12\\
				\le&C\left(\sum\limits_{n\neq 0}\left\|\sum\limits_{m\neq 0,n}v^1_{2,n-m}(\omega^1_m-\omega^2_m)\right\|_{L^2(\Omega)}^2
				+\sum\limits_{n\neq 0}\left\|\sum\limits_{m\neq 0,n}(v^1_{2,n-m}-v^2_{2,n-m})\omega^2_m\right\|_{L^2(\Omega)}^2\right)^\frac12\\
				\le&C\left\|\Q \Bv^1\cdot\Q\tilde{\omega}\right\|_{L^2(\Omega)}+C\left\|\Q \tilde{\Bv}\cdot\Q\omega^2\right\|_{L^2(\Omega)}\\
				\le&C\left\|\Q \Bv^1\right\|_{L^6(\Omega)}\left\|\Q\tilde{\omega}\right\|_{L^3(\Omega)}+C\left\|\Q \tilde{\Bv}\right\|_{L^6(\Omega)}\left\|\Q\omega^2\right\|_{L^3(\Omega)}\\
				\le & C\left\|\Q v^1\right\|_{H^\frac53(\Omega)}\left\|\Q\tilde{\Bv}\right\|_{H^\frac32(\Omega)}+C\left\|\Q v^2\right\|_{H^\frac53(\Omega)}\left\|\Q\tilde{\Bv}\right\|_{H^\frac32(\Omega)}\\
				\leq& C \Phi^\frac{1}{32}\left\|\Q\tilde{\Bv}\right\|_{H^\frac32(\Omega)}
		\end{aligned}\end{equation}
		and
		\begin{equation}\begin{aligned}\label{eq727}
				&\left(\sum\limits_{n\neq 0}\left\|\sum\limits_{m\neq 0,n}v^1_{1,n-m}\omega^1_{m}-v^2_{1,n-m}\omega^2_{m}\right\|_{L^2(\Omega)}^2\right)^\frac12\\
				\le&C\left\|\Q \Bv^1\cdot\Q\tilde{\omega}\right\|_{L^2(\Omega)}+C\left\|\Q \tilde{\Bv}\cdot\Q\omega^2\right\|_{L^2(\Omega)}\le C \Phi^\frac{1}{32}\left\|\Q\tilde{\Bv}\right\|_{H^\frac32(\Omega)}.
		\end{aligned}\end{equation}
		Similarly, one has
		\begin{equation}\begin{aligned}\label{eq728}
				&\left(\sum\limits_{n\neq 0}\left\|v^1_{2,n}\omega^1_{0}-v^2_{2,n}\omega^2_{0}\right\|_{L^2(\Omega)}^2\right)^\frac12\\
				\le&C\left(\sum\limits_{n\neq 0}\left\|v^1_{2,n}\tilde{\omega}_{0}\right\|_{L^2(\Omega)}^2\right)^\frac12+
				C\left(\sum\limits_{n\neq 0}\left\|\tilde{v}_{2,n}\omega^2_{0}\right\|_{L^2(\Omega)}^2\right)^\frac12\\
				\le&C\left\|\tilde{\omega}_{0}\right\|_{L^3(\Omega)}\left(\sum\limits_{n\neq 0}\left\|v^1_{2,n}\right\|_{L^6(\Omega)}^2\right)^\frac12+
				C\left\|\omega^2_{0}\right\|_{L^3(\Omega)}\left(\sum\limits_{n\neq 0}\left\|\tilde{v}_{2,n}\right\|_{L^6(\Omega)}^2\right)^\frac12\\
				\le&C\left\|\tilde{\Bv}_{0}\right\|_{H^2(\Omega)}\left(\sum\limits_{n\neq 0}\left\|\Bv^{1}_n\right\|_{H^\frac32(\Omega)}^2\right)^\frac12+
				C\left\|\Bv^{2}_{0}\right\|_{H^2(\Omega)}\left(\sum\limits_{n\neq 0}\left\|\tilde{\Bv}_n\right\|_{H^\frac32(\Omega)}^2\right)^\frac12\\
				\le&C \Phi^\frac{1}{32} \left(\left\|\tilde{\Bv}_0\right\|_{H^2(\Omega)}+\left\|\Q\tilde{\Bv}\right\|_{H^\frac32(\Omega)} \right).
		\end{aligned}\end{equation}
		Finally, it holds that
		\begin{equation}\begin{aligned}\label{eq729}
				\left(\sum\limits_{n\neq 0}\left\|v^1_{1,n}\omega^1_{0}-v^2_{1,n}\omega^2_{0}\right\|_{L^2(\Omega)}^2\right)^\frac12
				\le&C \Phi^\frac{1}{32}\left(\left\|\tilde{\Bv}_0\right\|_{H^2(\Omega)}+\left\|\Q\tilde{\Bv}\right\|_{H^\frac32(\Omega)} \right)
		\end{aligned}\end{equation}
		and
		\begin{equation}\begin{aligned}\label{eq730}
				&\left(\sum\limits_{n\neq 0}\left\|v^1_{1,0}\omega^1_{n}-v^2_{1,0}\omega^2_{n}\right\|_{L^2(\Omega)}^2\right)^\frac12\\
				\le&C\left\|v_{1,0}^1\right\|_{L^6(\Omega)}\left(\sum\limits_{n\neq 0}\left\|\tilde{\omega}_n\right\|_{L^3(\Omega)}^2\right)^\frac12+
				C\left\|\tilde{v}_{1,0}\right\|_{L^6(\Omega)}\left(\sum\limits_{n\neq 0}\left\|\omega^2_n\right\|_{L^3(\Omega)}^2\right)^\frac12\\
				\le&C \Phi^\frac{1}{32} \left(\left\|\tilde{\Bv}_0\right\|_{H^2(\Omega)}+\left\|\Q\tilde{\Bv}\right\|_{H^\frac32(\Omega)} \right).
		\end{aligned}\end{equation}
		
		Combining the estimates \eqref{eq724}-\eqref{eq730} gives
		\begin{equation}\nonumber
			\begin{aligned}
				\|\tilde{\Bv}_0\|_{H^2(\Omega)}+\|\Q\tilde{\Bv}\|_{H^\frac32(\Omega)}\le& C \Phi^{-\frac{11}{672}}\left(\|\tilde{\Bv}_0\|_{H^2(\Omega)}+\|\Q\tilde{\Bv}\|_{H^\frac32(\Omega)}\right).
	\end{aligned}\end{equation} }
	This implies that
	\begin{equation}\nonumber
		\begin{aligned}
			\|\tilde{\Bv}_0\|_{H^2(\Omega)}+\|\Q\tilde{\Bv}\|_{H^\frac32(\Omega)}=0,
	\end{aligned}\end{equation}
	as long as $\Phi\ge \tilde{C} (1+L)^{63}$ and $\tilde{C}$ is sufficiently large.
	
	It should be noted that the Sobolev embedding  constant $C$ which appear in \eqref{eq707}-\eqref{eq710}  and  \eqref{eq724}-\eqref{eq730} is independent of $L$.  The proof is similar to \eqref{largeL}.    Hence $\tilde{C}$ is also independent of $L$.
	
	{\em Step 5. The case $0< L <1$.}  When $0<L<1$, we choose some integer $M$ such that $1<ML<2$. Then one follows the proof of Steps 1-4 on the periodic domain $\tilde{\Omega}=\mathbb{T}_{2\pi M L}\times [-1,1]$ to conclude that the nonlinear problem \eqref{model11}-\eqref{model11'} admits a unique solution $\Bv\in H^2(\tilde{\Omega})$, provided $\Phi\ge \tilde{C}(1+L)^{63}$ and $\|\BF\|_{L^2(\tilde{\Omega})}\leq \Phi^\frac{1}{32}$. Since $\BF$ is periodic with period $2L$, so are $\Bv^j$ and $\Bv$. Moreover, $\Bv$  satisfies
	\begin{equation}\nonumber
		\|\Bv_0\|_{H^2(\Omega)}^2+ \sum\limits_{n\neq0}\left\|\Bv_n\right\|_{H^\frac53(\Omega)}^2 \leq \frac{1}{M}\left( \|\Bv_0\|_{H^2(\tilde{\Omega})}^2 + \sum\limits_{n\neq0}\left\|\Bv_n\right\|_{H^\frac53(\tilde{\Omega})}^2 \right)
		\le \frac{C}{M} \|\BF\|_{L^2(\tilde{\Omega})}^2 \leq C \|\BF\|_{L^2(\Omega)}^2,
	\end{equation}
	\begin{equation}\nonumber
		\|v_2\|_{L^2(\Omega)}^2 \leq \frac{1}{M} \|v_2\|_{L^2(\tilde{\Omega})}^2 \leq \frac{1}{M} \Phi^{-\frac{5}{6}}\leq L  \Phi^{-\frac56},
	\end{equation}
	and
	\begin{equation}\nonumber
		\| \Bv\|_{H^2(\Omega)}^2 \leq \frac{1}{M}\|\Bv  \|_{H^2(\tilde{\Omega})}^2 \le \frac{C}{M} \Phi^\frac12 \|\BF\|_{L^2(\tilde{\Omega})}^2
		\leq C \Phi^\frac12\|\BF\|_{L^2(\Omega)}^2 .
	\end{equation}
	Noting that  the assumption
	\begin{equation*}
		\|\BF\|_{L^2(\Omega)}\leq \sqrt{\frac{L}{2}}\Phi^\frac{1}{32} \leq M^{-\frac12}\Phi^\frac{1}{32}
	\end{equation*}
	implies $\|\BF\|_{L^2(\tilde{\Omega})}\leq \Phi^\frac{1}{32}$, we finish the proof of Proposition \ref{largeF}.
\end{proof}
Taking $\Bu=\Bv+\BU$, the proof of Theorem \ref{largeforce} is completed.


\section{Uniqueness of the solutions for the nonlinear problem}\label{sec-unique}
This section is devoted to the proof of Theorem \ref{uniqueness}. First we consider the case $L\geq 1$. 	
Let $\Bu$ be a  solution of the problem  \eqref{NS}-\eqref{BC} and $\Bv=\Bu-\BU$. Hence $\Bv$ satisfies the perturbed problem \eqref{model11}-\eqref{model11'}.
It  follows from the interpolation inequality and the linear estimates obtained in Propositions \ref{0mode}, \ref{mediumv}, and \ref{highv}  that one has
\begin{equation}\begin{aligned}\label{6-10}
		&~~\left\|\Bv_0\right\|_{H^2(\Omega)}+\Phi^\frac{1}{30}\left\|\Q\Bv\right\|_{H^\frac32(\Omega)}\\
		\leq&~~\left\|\Bv_0\right\|_{H^2(\Omega)}+C\Phi^\frac{1}{30}\left\|\Q\Bv\right\|_{L^2(\Omega)}^\frac{1}{10}\left\|\Q\Bv\right\|_{H^\frac53(\Omega)}^\frac{9}{10}\\
		\leq&~~C (\left\|(\Bv\cdot\nabla v_1)_0\right\|_{L^2(\Omega)}+\Phi^{-\frac{1}{70}}\left\|\Q(\Bv\cdot\nabla \Bv)\right\|_{L^2(\Omega)}),
\end{aligned}\end{equation}
where $\Q$ is the projection operator defined in \eqref{projection}.
Since $v_{2,0}=0$, using Hausdorff-Young inequality yields
\begin{equation}\label{6-4}
	\begin{aligned}
		\|(\Bv\cdot\nabla v_1)_0\|_{L^2}=&\|(v_1\partial_x v_1+v_2\partial_y v_1)_0\|_{L^2}\\
		=&\|-(v_1\partial_y v_2)_0+(v_2\partial_y v_1)_0\|_{L^2}\\
		\leq &~~ \|(v_1\partial_yv_2)_0\|_{L^2(\Omega)}+\|(v_2\partial_yv_1)_0\|_{L^2(\Omega)}\\
		\leq &~~ \left\|\sum_{n \neq0}v_{1,-n}v_{2,n}'\right\|_{L^2(\Omega)}+\left\|\sum_{n \neq0}v_{2,-n}v_{1,n}'\right\|_{L^2(\Omega)}\\
		\leq &~~ \|\Q v_1\|_{L^\infty(\Omega)}\|\Q(\partial_yv_2)\|_{L^2(\Omega)}+\|\Q v_2\|_{L^6(\Omega)}\|\Q(\partial_yv_1)\|_{L^3(\Omega)}\\
		\leq &~~ C\|v_2\|_{H^1(\Omega)}\|\Q \Bv\|_{H^\frac32(\Omega)}.
\end{aligned}\end{equation}
Moreover, with the aid of  the divergence free property of $\Bv$, one can rewrite the nonlinear term $\Bv\cdot\nabla\Bv$ as
\begin{equation}\nonumber
	\Bv\cdot\nabla\Bv=(-v_1\partial_yv_2+v_2\partial_yv_1,v_1\partial_xv_2+v_2\partial_yv_2)^T.
\end{equation}
Hence one has
\begin{equation}\label{6-5}\begin{aligned}
		\|\Q(\Bv\cdot\nabla\Bv)\|_{L^2(\Omega)}\leq &~~ \|\Q(v_1\partial_yv_2)\|_{L^2(\Omega)}+\|\Q(v_2\partial_yv_1)\|_{L^2(\Omega)}
		+\|\Q(v_1\partial_xv_2)\|_{L^2(\Omega)}\\
		&~~+\|\Q(v_2\partial_yv_2)\|_{L^2(\Omega)}.
\end{aligned}\end{equation}
By definition, one has
\begin{equation}\nonumber
	\begin{aligned}
		\|\Q(v_1\partial_yv_2)\|_{L^2(\Omega)}\leq&~~\left(\sum_{n\neq0}\left\|\sum_{m\in \Z}v_{1,n-m}v_{2,m}'\right\|_{L^2(\Omega)}^2\right)^\frac12.
\end{aligned}\end{equation}
Hence it holds that
\begin{equation}\nonumber
	\begin{aligned}
		\|\Q(v_1\partial_yv_2)\|_{L^2(\Omega)}  \leq&~~\left(\sum_{n\neq0}\left\|\sum_{m\neq 0,n}v_{1,n-m}v_{2,m}'\right\|_{L^2(\Omega)}^2\right)^\frac12
		+\left(\sum_{n\neq0}\left\|v_{1,0}v_{2,n}'\right\|_{L^2(\Omega)}^2\right)^\frac12\\
		\leq&~~\|\Q v_{1}\|_{L^\infty(\Omega)}\|\Q(\partial_yv_{2})\|_{L^2(\Omega)}
		+\|\Bv_0\|_{L^\infty(\Omega)}\|\Q(\partial_yv_{2})\|_{L^2(\Omega)}\\
		\leq&~~C\|v_2\|_{H^1(\Omega)} \left(\|\Bv_0\|_{H^2(\Omega)}+\|\Q\Bv\|_{H^\frac32(\Omega)}\right).
\end{aligned}\end{equation}
Since the other terms on the right-hand side of \eqref{6-5} can be estimated similarly, one finally obtains
\begin{equation}\label{6-7}
	\|\Q(\Bv\cdot\nabla\Bv)\|_{L^2(\Omega)}\leq C\| v_2\|_{H^1(\Omega)} \left(\|\Bv_0\|_{H^2(\Omega)}+\|\Q\Bv\|_{H^\frac32(\Omega)} \right).
\end{equation}

The estimate \eqref{6-10}, together with  \eqref{6-4}, \eqref{6-7}, and the assumption \eqref{est6-1}, yields
\begin{equation}\nonumber
	\begin{aligned}
		&~~\left\|\Bv_0\right\|_{H^2(\Omega)}+\Phi^\frac{1}{30}\left\|\Q\Bv\right\|_{H^\frac32(\Omega)}\\
		\leq&~~C \left[\Phi^\frac{1}{90}\left\|\Q\Bv\right\|_{H^\frac32(\Omega)}+\Phi^{-\frac{1}{315}}\left(\left\|\Bv_0\right\|_{H^2(\Omega)}
		+\left\|\Q\Bv\right\|_{H^\frac32(\Omega)} \right)\right].
\end{aligned}\end{equation}
This implies that
\begin{equation}\nonumber
	\left\|\Bv_0\right\|_{H^2(\Omega)}+\Phi^\frac{1}{30}\left\|\Q\Bv\right\|_{H^\frac32(\Omega)}=0,
\end{equation}
provided that $\Phi\ge \tilde{C}(1+L)^{63}$ and $\tilde{C}$ is sufficiently large. It is also noted that
all the constants $C$ and $\tilde{C}$ are independent of $L$ when $L\geq 1$.

For the case that $0<L<1$, we repeat the proof above on the periodic domain $\tilde{\Omega}=\mathbb{T}_{2\pi ML}\times [-1,1]$, where $M$ is an integer satisfying $1<ML<2$. Hence one could also obtain the uniqueness of solutions provided
$\|u_2\|_{H^2(\tilde{\Omega})}  \leq \Phi^{\frac{1}{90}}$.

Hence the proof of Theorem \ref{uniqueness} is completed.



\appendix

\section{Some elementary lemmas}\label{sec-app}
In this appendix, we collect some elementary lemmas which play important roles in the estimate needed in this paper and might be useful elsewhere. The proof for Lemmas \ref{lemmaA1}-\ref{lemmaHLP} can be found in \cite[Appendix A]{SWX}.
We first give some Poincar\'e type inequalities.
\begin{lemma}\label{lemmaA1}
	For a  function $g\in C^2([-1,1])$ satisfying $g(\pm1)=0$,
	it holds that
	\be \nonumber
	\int_{-1}^1 |g |^2  \, dy \leq  \int_{-1}^1 |g'|^2 \, dy.
	\ee
	Moreover, one has
	\be \nonumber
	\inte |g'|^2 \, dy \leq \left( \inte | g''|^2\, dy \right)^{\frac12} \left( \inte |g|^2 \, dy \right)^{\frac12}
	\ee
	and consequently,
	\be \nonumber
	\inte |g'|^2 \, dy
	\leq \inte | g''|^2  \, dy.
	\ee
\end{lemma}

In the following lemma, we give the pointwise estimate for the functions evaluated on the boundary.
\begin{lemma}\label{lemmaA2}
	For a function $g\in C^2([-1,1])$ with $g(\pm1)=0$, one has
	\begin{equation}\nonumber
		|g' (\pm1) |  \leq C\left(   \inte| g' |^2 \, dy \right)^{\frac14}  \left( \inte \left|g'' \right|^2  \, dy       \right)^{\frac14}.
	\end{equation}
\end{lemma}

	The following lemma is about a weighted interpolation inequality, which is quite similar to \cite[inequality (3.28)]{GHM}.
	\begin{lemma}\label{weightinequality} Let $g \in C^1([-1, 1])$, then one has
		\begin{equation}\nonumber
			\int_{-1}^1 |g|^2 \, dy  \leq C \left(\int_{-1}^1 (1-y^2)|g|^2  \, dy\right)^{\frac23} \left(\int_{-1}^1 |g'|^2 \, dy\right)^{\frac13}+C\int_{-1}^1 (1-y^2)|g|^2  \, dy.
		\end{equation}
	\end{lemma}

	The following lemma is a variant of Hardy-Littlewood-P\'olya type inequality \cite[p.165]{HLP}.
	\begin{lemma}\label{lemmaHLP}
		Let $g\in C^1([-1,1])$, one has
		\begin{equation}\nonumber
			\inte |g |^2\,dy\le \frac13\inte |g' |^2(1-y^2)\,dy +92\inte |g |^2(1-y^2)\,dy.
		\end{equation}
		
	\end{lemma}

		In fact, if $g$ is odd, we have the better estimate. The following lemma is inspired by \cite[Lemma 7]{Rabier1}.
		\begin{lemma}\label{lemmaA7}
			Let $g \in  C^1([-1, 1])$ be an odd function. For any $\delta>0$, there exists a constant $\delta_1=\min\{ \frac{1}{10}\delta,\frac13-\frac16\delta\}$ such that
			\begin{equation}\label{A7-1}
				\int_{-1}^1 |g |^2\,dy\leq \left(\frac12-\delta_1\right)\int_{-1}^1|g' |^2(1-y^2)\,dy+\delta \int_{-1}^1|g |^2(1-y^2)\,dy.
			\end{equation}
		\end{lemma}
		\begin{proof}

			Let
			\begin{equation*}
				P_0=\frac{1}{\sqrt2},~P_1=\sqrt{\frac32}y,\cdots
			\end{equation*}
			be the normalized Legendre polynomials on the interval $[-1,1]$, which satisfies the differential equation
			\begin{equation*}
				((1-y^2)P_m')'=\lambda_m P_m
			\end{equation*}
			with $\lambda_m=-m(m+1)$. Since $g$ is an odd function, one can write
			\begin{equation*}
				g=\sum_{m=1}^\infty C_{m}P_m.
			\end{equation*}
			Then it follows from Parseval's identity that one has
			\begin{equation}\label{A7-1-1}
				\int_{-1}^1 |g |^2\,dy=  \sum_{m=1}^\infty C_{m}^2.
			\end{equation}
			Using Minkowski inequality and Young's inequality, one has
			\begin{equation}\label{A7-1-2}
				\begin{aligned}
					&\int_{-1}^1|g |^2(1-y^2)\,dy\\
					\ge &\left[\left(\int_{-1}^1C_1^2P_1^2(1-y^2)\,dy\right)^\frac12-\left(\int_{-1}^1|g-C_1P_1|^2(1-y^2)\,dy\right)^\frac12\right]^2\\
					=&\int_{-1}^1C_1^2P_1^2(1-y^2)\,dy+\int_{-1}^1|g-C_1P_1|^2(1-y^2)\,dy\\
					&- 2\left(\int_{-1}^1C_1^2P_1^2(1-y^2)\,dy\right)^\frac12\left(\int_{-1}^1|g-C_1P_1|^2(1-y^2)\,dy\right)^\frac12\\
					\ge& \frac12\int_{-1}^1C_1^2P_1^2(1-y^2)\,dy-\int_{-1}^1|g-C_1P_1|^2(1-y^2)\,dy.
				\end{aligned}
			\end{equation}
			Combining \eqref{A7-1-1} and \eqref{A7-1-2} gives
			\begin{equation}\label{A7-3}
				\begin{aligned}
					& \int_{-1}^1 |g |^2\,dy-\delta \int_{-1}^1|g |^2(1-y^2)\,dy\\
					\leq & \sum_{m=1}^\infty C_{m}^2 -\frac12\delta \int_{-1}^1C_1^2P_1^2(1-y^2)\,dy+\delta \int_{-1}^1|g-C_1P_1|^2(1-y^2)\,dy.
				\end{aligned}
			\end{equation}
			Hence it follows from straightforward computations that
			\begin{equation}
				\begin{aligned}
					& \int_{-1}^1 |g |^2\,dy-\delta \int_{-1}^1|g |^2(1-y^2)\,dy\\
					\leq & \sum_{m=1}^\infty C_{m}^2 -\frac34\delta C_1^2\int_{-1}^1y^2(1-y^2)\,dy+\delta \int_{-1}^1|g-C_1P_1|^2\,dy\\
					\leq &\sum_{m=1}^\infty C_{m}^2 -\frac15\delta C_1^2+\delta \sum_{m=2}^\infty C_{m}^2\\
					= &\left(1-\frac15\delta\right)C_1^2+ (1+\delta)\sum_{m=2}^\infty C_{m}^2.
				\end{aligned}
			\end{equation}
			
			On the other hand, one uses integration by parts to obtain
			\begin{equation*}
				\int_{-1}^1|g' |^2(1-y^2)\,dy=-\int_{-1}^1((1-y^2)g' )'\overline{g }\,dy=-\sum_{m=1}^\infty \lambda_m C_{m}^2.
			\end{equation*}
			Noting $-\lambda_1=2$ and $-\lambda_m\ge 6$ for any $m\ge 2$, one has
			\begin{equation}\label{A7-4}
				\int_{-1}^1|g' |^2(1-y^2)\,dy\ge 2 C_1^2+6\sum_{m=2}^\infty C_{m}^2.
			\end{equation}
			
			Set $\delta_1=\min\{ \frac{1}{10}\delta,\frac13-\frac16\delta\}$, then one has
			\begin{equation}\nonumber
				\begin{aligned}
					2\left(\frac12 -\delta_1\right)\ge 1-\frac{1}{5}\delta\quad \text{and}\quad
					6\left(\frac12 -\delta_1\right)\ge (1+\delta).
				\end{aligned}
			\end{equation}
			This, together with \eqref{A7-3} and \eqref{A7-4}, gives \eqref{A7-1}. Hence  the proof of this lemma is completed.
			
		\end{proof}

		The following lemma on the Airy function gives the estimate for the boundary layers constructed in Section \ref{secinter}.
		\begin{lemma}\label{airy-est}
			(1) Let $C_{0, n,\Phi}$  be given in \eqref{5-3-62} and $G_{n,\Phi}$ be defined in \eqref{5-3-61}.  It holds that
			\begin{equation}\nonumber
				\tilde{C}_{0}:=\inf \left\{\left|C_{0, n, \Phi}\right|:\ \Phi\ge  1,1\le |n| \leq L \Phi^{\frac{1}{2}}\right\}>0
			\end{equation}
			and
			\begin{equation}\nonumber
				\sup _{\Phi \geq 1} \sup _{1\le |n| \leq L \Phi^{\frac{1}{2}}}  \sup _{\rho \ge R} e^{\rho}\left|\frac{d^k G_{n, \Phi}}{d \rho^k}(\rho)\right|<\infty, ~~k=0,1,2,3,
			\end{equation}
			provided that the constant $R$ is sufficiently large.
			
			(2) There exists a constant $\epsilon\in (0,1)$ such that defining
			\begin{equation}\nonumber
				\Sigma_{\epsilon}:=\left\{\mu \in \mathbb{C} | \text { arg } \mu=-\frac{\pi}{6}, 0 \leq|\mu| \leq \epsilon\right\},
			\end{equation}
			then
			\begin{equation}\nonumber
				K_{\epsilon}:=\inf _{\mu \in \Sigma_{\epsilon}}\left|\int_\ell e^{-\mu z} Ai\left(z+\mu^{2}\right) \,dz\right|\ge\frac16,
			\end{equation}
			where $\ell$ is the contour $\ell:=\left\{re^{i\frac{\pi}{6}}|r\ge 0\right\}$.
		\end{lemma}
		Since the proof of Lemma \ref{airy-est} is exactly the same as that of \cite[Lemma 3.7]{GHM}, we omit the proof here.


		
		\medskip
		
		{\bf Acknowledgement.}
		The research of Wang was partially supported by NSFC grants 12171349 and 12271389.  The work of Xie was partially supported by NSFC grants 11971307 and 1221101620,  Natural Science Foundation of Shanghai 21ZR1433300, and Program of Shanghai Academic Research Leader 22XD1421400. The authors would like to thank Professor Congming Li for helpful discussions.


\begin{thebibliography}{99}
			\bibitem{Adams}
			R. A.  Adams and J. Fournier, {\it  Sobolev spaces}, Second edition, Elsevier/Academic Press, Amsterdam, 2003.	
			
			\bibitem{AP}
			K. A. Ames and L. E. Payne,  Decay estimates in steady pipe flow, {\it SIAM J. Math. Anal.}, {\bf 20} (1989), 789--815.
			
			\bibitem{A1} C. J. Amick, Steady solutions of the Navier-Stokes equations in unbounded channels and pipes, {\em Ann. Scuola Norm. Sup. Pisa Cl. Sci.}, {\bf 4}(1977), 473--513.
			
			\bibitem{A2} C. J. Amick, Properties of steady Navier-Stokes solutions for certain unbounded channels and pipes, {\em Nonlinear Anal.}, {\bf 2}(1978), 689--720.
			
			\bibitem{AF} C. J. Amick and L. E. Fraenkel, Steady solutions of the Navier-Stokes equations representing plane flow in channels of various types, {\em Acta Math.}, {\bf 144}(1980), 83--151.
			
			\bibitem{Conti}
			M. Coti Zelati, T.  Elgindi, and K. Widmayer,  Enhanced dissipation in the Navier-Stokes equations near the Poiseuille flow, {\em Comm. Math. Phys.}, {\bf 378} (2020), 987--1010.
			
			\bibitem{CWZ} Q. Chen, D. Wei and Z. Zhang, Linear stability of pipe Poiseuille flow at high Reynolds number regime, arXiv:1910.14245.
			
			\bibitem{Ding1}
			S. Ding, Q. Li, and Z.  Xin,  Stability analysis for the incompressible Navier-Stokes equations with Navier boundary conditions, {\em J. Math. Fluid Mech.}, {\bf 20} (2018), no. 2, 603--629.
			
			\bibitem{Ding2} S. Ding and Z. Lin, Enhanced dissipation and transition threshold for the 2-D plane Poiseuille flow via resolvent estimate,
			{\em J. Differential Equations} {\bf 332} (2022), 404--439.
			
			\bibitem{DR} P. G. Drazin and W. H. Reid, {\em Hydrodynamic Stability}, Cambridge University Press, Cambridge, New York, 1981.
			
			\bibitem{Galdi} G. P. Galdi, {\em An introduction to the mathematical theory of the Navier-Stokes equations: Steady-state problems}. Springer, New-York, 2011.
			
			\bibitem{GGN} E. Grenier, Y. Guo and T. Nguyen, Spectral instability of general symmetric shear flows in a two-dimensional channel, {\em  Adv. Math.}, {\bf 292}(2016), 52--110.
			
			\bibitem{GHM} I. Gallagher, M. Higaki and Y. Maekawa, On stationary two-dimensional flows around a fast rotating disk, {\it Math. Nachr.}, {\bf 292} (2019), 273--308.
			
			\bibitem{GT}  D. Gilbarg and N. Trudinger, {\em Elliptic partial differential equations of second order}. Springer-Verlag, Berlin, 1998.
			
			
			\bibitem{HLP} G. H. Hardy, J. E. Littlewood, and G. P\'{o}lya, {\em Inequalities, 2nd edition}. Cambridge University Press, 1952.
			
			
			\bibitem{Hopf}
			E.  Hopf, \"{U}ber die Anfangswertaufgabe f\"{u}r die hydrodynamischen Grundgleichungen, {\it  Math. Nachr.},  {\bf 4} (1951), 213--231.
			
			\bibitem{Horgan}
			C. O. Horgan,  Plane entry flows and energy estimates for the Navier-Stokes equations, {\it Arch. Rational Mech. Anal.}, {\bf  68} (1978), 359--381.
			
			\bibitem{HW}
			C. O. Horgan and L. T.  Wheeler, Spatial decay estimates for the Navier-Stokes equations with application to the problem of entry flow, {\it SIAM J. Appl. Math.}, {\bf 35} (1978), 97--116.
			
			
			\bibitem{KP}
			L. V. Kapitanski and K. I. Piletskas, Spaces of solenoidal vector fields and boundary value problems for the Navier-Stokes equations in domains with noncompact boundaries. (Russian) Boundary value problems of mathematical physics, 12. Trudy Mat. Inst. Steklov. {\bf 159} (1983), 5--36.
			
			\bibitem{KPR}
			M. V. Korobkov, K. Pileckas and R. Russo, Solution of Leray's problem for stationary Navier-Stokes equations in plane and axially symmetric spatial domains, {\em Annals of Mathematics}, {\bf 181} (2015), 1--39.
			
			
			\bibitem{LS} O. A. Ladyzhenskaja and V. A. Solonnikov, Determination of solutions of boundary value problems for stationary Stokes and Navier-Stokes equations having an unbounded Dirichlet integral, {\em Zap. Nauchn. Sem. Leningrad. Otdel. Mat. Inst. Steklov. (LOMI)}, {\bf 96} (1980), 117--160.
			
			\bibitem{Lax}
			{P. D. Lax},   {Functional analysis}, Pure and Applied Mathematics (New York), Wiley-Interscience [John Wiley \& Sons], New York, 2002.
			
			\bibitem{Leray} J. Leray,  \'{E}tude de Diverses \'{E}quations Int\'{e}grales non Lin\'{e}aires et de Quelques Probl\'{e}mes que Pose l'Hydrodynamique, {\em J. Math. Pures Appl.}, {\bf 12} (1933), 1--82.
			
			\bibitem{L1} C. C. Lin, {\em The Theory of Hydrodynamic Stability}. Cambridge University Press, 1955.
			
			\bibitem{Morimoto}
			H. Morimoto,  Stationary Navier-Stokes flow in 2-D channels involving the general outflow condition. Handbook of differential equations: stationary partial differential equations. Vol. IV, 299--353, Handb. Differ. Equ., Elsevier/North-Holland, Amsterdam, 2007.
			
			\bibitem{MF}
			H. Morimoto and H. Fujita,  On stationary Navier-Stokes flows in 2D semi-infinite channel involving the general outflow condition. Navier-Stokes equations and related nonlinear problems (Ferrara, 1999). {\it Ann. Univ. Ferrara Sez. VII (N.S.)}, {\bf 46} (2000), 285--290.
			
			\bibitem{NP1}
			S. A. Nazarov and K. I. Piletskas, Behavior of solutions of Stokes and Navier-Stokes systems in domains with periodically changing cross-section. (Russian) Boundary value problems of mathematical physics, 12. Trudy Mat. Inst. Steklov., {\bf 159} (1983), 95--102.
			
			\bibitem{NP2}
			S. A. Nazarov and K. I. Piletskas, The Reynolds flow of a fluid in a thin three-dimensional channel, Litovsk. Mat. Sb., {\bf 30} (1990), no. 4, 772--783.
			
			\bibitem{NP3}
			S. A. Nazarov and K. I. Piletskas,  On the solvability of the Stokes and Navier-Stokes problems in the domains that are layer-like at infinity, {\it  J. Math. Fluid Mech.}, {\bf 1} (1999),  78--116.
			
			\bibitem{Orszag}
			S. A. Orszag,  Accurate solution of the Orr-Sommerfeld stability equation, {\it J. Fluid
				Mech.},  {\bf 50} (1971), 689--703.
			
			
			\bibitem{Rabier1}
			P. J. Rabier,  Invertibility of the Poiseuille linearization for stationary two-dimensional channel flows: symmetric case, {\it J. Math. Fluid Mech.}, {\bf 4} (2002),  327--350.
			
			
			\bibitem{Rabier2}
			P. J. Rabier,  Invertibility of the Poiseuille linearization for stationary two-dimensional channel flows: nonsymmetric case, {\it J. Math. Fluid Mech.}, {\bf 4} (2002), 351--373.
			
			\bibitem{SWX} K. Sha, Y. Wang and C. Xie, Uniform structural stability and uniqueness of Poiseuille flows in a two dimensional periodic strip, 2020, arXiv:2011.07467.
			
			\bibitem{WX1} Y. Wang and C. Xie, Uniform Structural stability of Hagen-Poiseuille flows in a pipe, {\em Comm. Math. Phys.}, {\bf 393} (2022), no. 3, 1347--1410.
			
			\bibitem{WX2} Y. Wang and C. Xie, Existence and asymptotic behavior of large axisymmetric solutions for steady Navier-Stokes system in a pipe,  {\em Arch. Ration. Mech. Anal. }, {\bf 243} (2022), no. 3, 1325--1360.
			
			\bibitem{Youdovich} V. I. Youdovich, An example of loss of stability and generation of secondary flow in a closed vessel, {\em Mat. Sb.}, {\bf 74}(1967), 306--329.
		\end{thebibliography}
	\end{document}